\documentclass
{article}

\usepackage{latexsym}
\usepackage{amsmath}
\usepackage{amsfonts}
\usepackage{amsthm}
\usepackage{amssymb}

\usepackage{tikz}

\usepackage{color}

\newtheorem{theorem}{Theorem}[section]

\newtheorem{claim}[theorem]{Claim}
\newtheorem{lemma}[theorem]{Lemma}
\newtheorem{problem}[theorem]{Problem}
\newtheorem{proposition}[theorem]{Proposition}
\newtheorem{corollary}[theorem]{Corollary}
\newtheorem{question}[theorem]{Question}

\newtheorem{fact}[theorem]{Fact}
\newtheorem{remark}[theorem]{Remark}
\newtheorem{em-example}[theorem]{Example}
\newtheorem{em-def}[theorem]{Definition}
\newtheorem{conjecture}[theorem]{Conjecture}

\newenvironment{example}{\begin{em-example} \em }{ \end{em-example}}
\newenvironment{definition}{\begin{em-def} \em  }{ \end{em-def}}

\DeclareMathOperator{\GL}{GL}

\DeclareMathOperator{\ke}{ker}

\DeclareMathOperator{\cof}{cofin}

\newcommand{\mb}{$\blacklozenge\ $}

\usepackage{mathrsfs}
\newcommand{\nilpo}{\mathscr N}

\def\AA{{\mathfrak A}}

\newcommand\Mar{\mathfrak M}
\def\T{{\mathbb T}}

\def\Q{{\mathbb Q}}
\def\Z{{\mathbb Z}}
\def\N{{\mathbb N}}

\def\C{{\mathcal C}}

\def\cont{\mathfrak{c}}

\def\C{{\mathscr C}}
\def\I{{\mathscr I}}
\def\Po{{\mathscr P}}
\def\M{{\mathscr M}}

\def\Zar{\mathfrak Z}

\def\Mar{\mathfrak M}
\def\Zar{\mathfrak Z}

\newcommand{\zar}{\mathfrak Z}
\newcommand{\mar}{\mathfrak M}

\newcommand{\pz}{\mathfrak Z}

\author{Marco Bonatto\footnote{ Department of Mathematics and Computer Science, University of Ferrara, Via Macchiavelli 30, 44121 Ferrara, Italy, {\tt marco.bonatto.87@gmail.com}}, Dikran Dikranjan\footnote{Department of Mathematics, Computer Science and Physics, University of Udine, Via Delle Scienze 206, 33100 Udine, Italy, {\tt dikran.dikranjan@uniud.it}}, 
Daniele Toller\footnote{Department of Computer Science, AAlborg Universitet, Selma Lagerløfs Vej 300, 9220 Aalborg, Denmark, {\tt tollerdaniele@gmail.com}}}

\title{Groups with cofinite Zariski topology and potential density
\thanks{
MSC: 
20E28,   
20F18, 		
08A50,		
22A05,		
54H11		
\endgraf
Keywords: topologizable group, non-topologizable group, Noetherian space, Zariski topology, Markov topology, centralizer topology, 
monomial topology, independent group topologies, unconditionally closed subset, potentially dense subset
}
}

\date{Respectfully dedicated to the memory of 
 Mitrofan Choban}

\begin{document}

\maketitle

\begin{abstract}
Tkachenko and Yaschenko \cite{TY} characterized the abelian groups $G$ such that all proper unconditionally closed subsets of
$G$ are finite, these are precisely the abelian groups $G$ having cofinite Zariski topology (they proved that 
such a $G$ is either almost torsion-free or of prime exponent). The authors  connected this fact to Markov's notion of potential 
density and the existence of pairs of independent group topologies. 
Inspired by their work, we examine the class $\C$ of groups having cofinite Zariski topology 
in the general case, obtaining a number of very strong restrictions on these groups
in the non-abelian case which suggest the bold conjecture that a group with cofinite Zariski topology is necessarily either abelian
or finite.
We show that  Tkachenko--Yaschenko theorem fails in the non-abelian case and we offer a natural counterpart 
 in the general case using a partial Zariski topology and an appropriate stronger version of the property almost torsion-free.
\end{abstract}


\section{Introduction}

We denote by $\N$ the set of naturals, by $\N_+$ the set of positive integers and by $\Z$ the group of integers.
If $X$ is a set, the symbol $\cof_X$ will denote the cofinite topology on $X$. 

For a group $G$ the neutral element of $G$ is denoted by $e_G$, or simply by $e$.
For $n \in \N_+$, let $G[n] = \{ g \in G \mid g^n = e\} \subseteq G$ denote the \emph{$n$-socle} of $G$. 
The elements of $t(G) = \bigcup_{n\in \N_+} G[n]$ are called \emph{torsion elements}; 
$G$ is called \emph{torsion-free} (resp.,  \emph{almost torsion-free}) if $G[n] = \{e\}$ (resp., $|G[n]|<\infty$) for every $n\in \N_+$. If $G$ is abelian, $G[n]$ is a subgroup of $G$, and $G$ is almost torsion-free if and only if $G[p]$ is finite for every prime number $p$. A group $G$ is said to be of  {\em finite exponent} (or, {\em bounded}) when $G=G[m]$ for some $m\in \N_+$, 
the smallest $m$ with this property is called exponent of $G$. 
The remaining notation and terminology from group theory used throughout the paper can be found either in \S \ref{N&T} or \cite{Robinson}. 

\subsection{Markov's problems and two topologies on infinite groups}
The following problem was raised by Markov in 1944: 
 
\begin{problem}\label{Markov's:problem}
 Does there exist an infinite non-topologizable (i.e., admitting no non discrete Hausdorff group topology) group?
\end{problem}

Moreover, Markov introduced four special families of subsets of a group $G$:  

\begin{definition}\cite{Markov1} \label{Markov's:definition}
A subset $X$ of a group $G$ is called:
\begin{itemize}
  \item[(a)] {\em elementary algebraic} if there exist an integer $n>0$, elements $g_1,\ldots, g_n\in G$ and $\varepsilon_1,\ldots,\varepsilon_n \in\{-1,1\}$, such that 		
  $X=\{x\in G: g_1 x^{\varepsilon_1} g_2 x^{\varepsilon_2} \cdots g_n x^{\varepsilon_n} = e_G\}$;
  \item[(b)] 
{\em algebraic} if $X$ is an intersection of  finite unions of algebraic subsets of $G$;
  \item[(c)]   {\em unconditionally closed} if $X$ is closed in {\em every\/} Hausdorff group topology on $G$;
   \item[(d)] {\em potentially dense} if $X$ is dense in {\em some\/} Hausdorff group topology on $G$.
\end{itemize}
\end{definition}

Obviously, (a) $\to$ (b) $\to$ (c), yet the question of whether (c) $\to$ (a) always holds true remained open: 

\begin{problem}\label{Markov's:problem*}
 Is (c) $\to$ (a) always true?
\end{problem}

On the other hand, a proper unconditionally closed set cannot be potentially dense. Markov proved that every infinite subset of $\Z$ is potentially dense
and posed the problem to characterize the potentially dense subsets in general. 

As pointed out in \cite{Selected} and \cite{Reflection}, one can more conveniently face these two problems by introducing two topologies as follows. 
Every singleton is an elementary algebraic subset, so every finite subset is algebraic. We denote by $\mathfrak E_G$ the family of elementary algebraic subsets of $G$. The family $\AA_G$ of all algebraic sets of $G$ contains all finite subsets of $G$ and $\AA_G$ is closed under taking finite unions and arbitrary intersections. So, ${\mathfrak A}_G$ is the family of closed sets of a $T_1$ topology ${\mathfrak Z}_G$ on $G$, named {\em Zariski topology of $G$} (\cite{Selected,Reflection}). 
Markov  did not explicitly introduce this topology, although it was implicitly present in \cite{Markov1}. It was  explicitly introduced only in 1977 by Bryant \cite{Bryant}  under the name \emph{verbal topology}. 

The Zariski topology has a very transparent description for abelian groups, since 
for an abelian group $G$  one has $\AA_G = \{0\}\cup \{ g + G[n] \mid g \in G, n \in \N \}$ (see Proposition \ref{Abel:Mon} and \cite{Articolone} for further properties). 

On the other hand, the family of unconditionally closed subsets of $G$ coincides with the family of closed subsets of a $T_1$ topology $\mar _G$ on $G$, called \emph{the Markov topology of $G$} in \cite{Selected,Reflection}. It coincides with the intersection of all Hausdorff group topologies on $G$, and the implication (b) $\to$ (c) above means that $\mathfrak Z_G \leq \mathfrak M_G$ for every group $G$. In these terms, Problem \ref{Markov's:problem*} reduces to the simple equality $\mathfrak Z_G = \mathfrak M_G$. 
Markov \cite{Markov1} positively answered Problem \ref{Markov's:problem*} for countable groups (i.e., by proving that $\mathfrak Z_G = \mathfrak M_G$ for any  countable group $G$, although none of these two topologies explicitly appeared in his paper). 
In the same paper, he asked whether this equality remains true for an arbitrary group.  This remained an open problem for decades, until 1979 Hesse's negative solution in \cite{Hesse}, where he constructed examples of arbitrarily large groups $G$ such that $\mathfrak M_G$ is discrete, but $\mathfrak Z_G$ is non-discrete.  Shakhmatov and the second named author \cite{Reflection} proved that ${\mathfrak Z}_G = {\mathfrak M}_G$ for abelian groups (see also \cite{Sip2}).
Banakh, Guran and  Protasov \cite{BGP} made a major progress towards understanding the equality ${\mathfrak Z}_G = {\mathfrak M}_G$, by proving that it holds in all infinite permutation groups $G = S(X)$.
Recently,  Shakhmatov and Yañez \cite{SY} established the equality $\mathfrak Z_F = \mathfrak M_F$ for free non-abelian groups $F$. 

These topologies are very useful for understanding and treating also Problem \ref{Markov's:problem}. A group $G$ is called {\em ${\mathfrak Z}$-discrete} (resp., 
{\em ${\mathfrak Z}$-Haudorff}) if $(G,{\mathfrak Z}_G)$ is discrete (resp., Hausdorff). Similarly {\em ${\mathfrak M}$-discrete} ({\em ${\mathfrak M}$-Haudorff}) groups are defined, and this terminology is similarly extended to other topological properties. In these terms, the  ${\mathfrak M}$-discrete groups are precisely the non-topologizable groups. Since  ${\mathfrak Z}_G \leq {\mathfrak M}_G$ holds for every group $G$, ${\mathfrak Z}$-discrete groups are  non-topologizable and this was used by Ol$'$shankii to build the first 
example of a countable non-topologizable group (more information and reference on this topic can be 
found in the surveys \cite{Selected,survey,secondo:survey}, along with many examples). 


The utility of these topologies for the solution of Markov's problem on potential density will be discussed in \S 1.3. 
For further applications of these topologies see \cite{DS_ConnectedMarkov}
(for the solution of the fourth Markov problem of connected topologization of abelian groups) and \cite{DS:final} (for the solution of Protasov-Comfort's problem on minimally almost periodic group topologies).



\subsection{$\mathfrak Z$-Noetherian groups}\label{NoetherianGettho}

A topological space $X$ is {\em Noetherian}, if $X$ satisfies the ascending chain condition on open sets.
Obviously, a Noetherian space is compact, and a subspace of a Noetherian space is Noetherian itself. Actually, a space is Noetherian if and only if all its subspaces are compact. Hence, an infinite Noetherian space is never Hausdorff.

 Let $K$ be a field, and $d$ be a positive integer.  Consider the finite dimensional vector space $K^d$, and the ring $K[x_1, \ldots, x_d]$ of polynomials in $d$ variables with coefficients in the field $K$. Recall that the family of zero-sets of those polynomials is a subbase of the closed sets of a Noetherian $T_1$ topology $\mathcal A_{K^d}$ on $K^d$, usually called the {\em affine topology} of $K^d$.
If $X$ is a subset of $K^d$, the affine topology of $X$ is defined as $\mathcal A_X = \mathcal A_{K^d} \restriction _X$.
In particular, the full linear group $\GL_n(K)$ and all its subgroups carry the topology induced by $\mathcal A_{K^{n^2}}$ (via the embedding in $K^{n^2}$). 
It is well known that  the affine topology $\mathcal A_{K^d}$ is Noetherian. 
%

Bryant \cite{Bryant} studied first the class $\nilpo$ of $\zar$-Noetherian groups, under the name `\emph{groups which satisfy \emph{min}-closed}' (i.e. the minimal condition on Zariski closed sets). He proved that the classes of groups in the following example are $\mathfrak Z$-Noetherian.

\begin{example}   \label{abeliani e lineari sono ZNoet}
\begin{itemize}
\item[(a)] \cite[Theorem 3.5]{Bryant} If $G$ is a linear group, then $\mathfrak Z_G \subseteq \mathcal A _G$.  In particular, $G$ is $\mathfrak Z$-Noetherian. 
\item[(b)]  \cite[Corollary 3.7]{Bryant} Every finitely generated, abelian-by-nilpotent-by-finite group is $\mathfrak Z$-Noetherian.
\item[(c)] \cite[Theorem 3.8]{Bryant} Every abelian-by-finite group is $\mathfrak Z$-Noetherian.
\end{itemize}
\end{example}

Bryant then proved that the class of $\zar$-Noetherian groups is stable under taking subgroups, and under taking finite products:
\begin{fact} \label{prod fin di Znoet}\label{ZNoet ereditaria}
\begin{itemize}
\item[(a)]  \emph{\cite[Lemma 3.3]{Bryant}} Every subgroup of a $\mathfrak Z$-Noetherian group is $\mathfrak Z$-Noetherian.
\item[(b)]  \emph{\cite[Lemma 3.4]{Bryant}} The finite product of $\mathfrak Z$-Noetherian groups is a $\mathfrak Z$-Noetherian group.
\end{itemize}
\end{fact}

In other words, a finite direct product $G = \prod_{i = 1}^n G_i$ is $\mathfrak Z$-Noetherian if and only if for every $i = 1, \ldots, n$ the group $G_i$ is $\mathfrak Z$-Noetherian. These results were  extended as follows in \cite{survey}.
  
\begin{theorem}\label{group:Z:Noeth:iff:every:countable:subgroup:is}\label{Product:Z:Noeth:iff:} 
\begin{itemize}
\item[(a)]\cite[Theorem 5.5]{survey} A group $G$ is $\mathfrak Z$-Noetherian if and only if every countable subgroup of $G$ is $\mathfrak Z$-Noetherian.
\item[(b)]\cite[Theorem 5.7]{survey}  Let $\{ G_i \mid i \in I\}$ be a family of groups, and $G = \prod_{i \in I} G_i$. Then $G$ is $\mathfrak Z$-Noetherian if and only if every $G_i$ is $\mathfrak Z$-Noetherian and all but finitely many of the groups $G_i$ are abelian.
\end{itemize}
\end{theorem}

As a corollary of Theorem \ref{group:Z:Noeth:iff:every:countable:subgroup:is}(a), one obtains that every free group is $\mathfrak Z$-Noetherian (see also \cite{ChiswRemesl}, \cite{Guba}, and Example \ref{zariski topology on a free group}(a)). 


\subsection{Independent topologies, potential density and $\mathfrak M$-cofinite groups}\label{TY-story}

The following definition was given in \cite[Definition 2.13]{Zariski_linear}:

\begin{definition}\label{Def:Zar-cofin}\label{definition:Z:Noetherian:etc}
A group $G$ is called \emph{$\zar$-cofinite} (resp., \emph{$\Mar$-cofinite}) if $\mathfrak{Z}_{G} = \cof_G$ (resp., $\mathfrak{M}_{G}$ is cofinite).
\end{definition}

Obviously, finite groups are $\zar$-cofinite, while $\zar$-cofinite implies $\zar$-Noetherian.
The aim of the present paper is the study of the smaller class $\C\subseteq \nilpo$ of $\zar$-cofinite groups. 
A examples of groups that are $\zar$-Noetherian but not $\zar$-cofinite are provided by 
Example \ref{abeliani e lineari sono ZNoet} (we show in Corollary \ref{cor:example:not:zar:cofinite} that they are $\zar$-cofinite only if they are either finite or abelian), 
for further examples see Example \ref{zariski topology on a free group}).


 Since  ${\mathfrak Z}_G \leq {\mathfrak M}_G$, the class $\M$ of $\Mar$-cofinite groups is  contained in $\C$, and clearly $G \in \M$ precisely when all proper unconditionally closed sets of $G$ are finite. The abelian ${\mathfrak M}$-cofinite groups
were described by Tkachenko and Yaschenko: 

\begin{theorem}{\em \cite[Theorem 5.1]{TY}}\label{abelian:Z:cofinite}
In an abelian group $G$ all proper unconditionally closed subgroups of G are finite if and only if $G$ is almost torsion-free, or it has prime exponent.
\end{theorem}

Since the Markov and the Zariski topology coincide in abelian groups, this theorem characterizes the abelian groups having cofinite Zariski topology. 

Even if the above observations (see also \S \ref{NoetherianGettho}) may fully justify the study of the $\zar$-cofinite and $\mar$-cofinite groups, an even stronger motivation comes 
from the problem of complementation of group topologies. More specifically, pairs of independent group topologies
(two non-discrete $T_1$-topologies $\tau_1, \tau_2$ on a set X are called {\em independent} if $\tau_1\cap \tau_2=\cof_X$ \cite{TY}).
Let $\I$ denote the class of groups admitting a  pair of independent group topologies. Obviously, $\I \subseteq \M$ 
 as $\tau_1\cap \tau_2 \geq {\mathfrak M}$ for every group $G$.  Yet these two classes do not coincide, 
 as $\I$ contains no countable groups (\cite[Theorem 2.3]{TY}),  while Theorem \ref{abelian:Z:cofinite} provides a wealth of countable abelian groups in $\M$. 
Nevertheless, under the assumption of MA (the Martin Axiom) one has:

\begin{theorem}\label{TY:Ind/Mar-cofin} {\em \cite[Corollary 5.4]{TY}} \ {\rm [MA]} An Abelian group $G$ of size $\mathfrak c$ belongs to $\I$ if and only if $G\in \M$. 
\end{theorem}

Every potentially dense set is also $\mathfrak M$-dense and hence $\mathfrak Z$-dense. So a possible solution of Markov's problem on characterization of the potentially dense subsets of a group should involve $\mathfrak M$-density. 
By Theorem \ref{abelian:Z:cofinite}, every infinite subset of an almost torsion-free abelian group is ${\mathfrak M}$-dense.  
Let $\Po$ denote the class of groups where every infinite subset is potentially dense. Then $\Po \subseteq \M$ by what we said above. 
A partial inverse of this implication was proved in \cite[Theorem 3.10]{TY} as follows: if $G$ is an almost torsion-free abelian group 
(so $G\in \M$) and $|G|\leq \mathfrak c$, then $G \in \Po$.
 These authors asked whether the condition $|G|\leq \mathfrak c$ can be relaxed to $|G|\leq 2^\mathfrak c$ (\cite[Problem 6.5]{TY}). The necessity of 
the restraint $|G|\leq 2^\mathfrak c$ is due to the fact that if $G \in \Po$, then 
 its countably infinite subsets are potentially dense, so $G$ is separable in the Hausdorff group topology witnessing potential density
of any of these countably infinite subsets, so $|G|\leq 2^\mathfrak c$.

The above question of whether all abelian groups $G\in \M$ with $|G|\leq 2^\mathfrak c$ belong to $\Po$ 
was answered in a strongly positive way (in several directions) by Shakhmatov and the second named author \cite{Kronecker-Weyl}. First of all, they proved the following simultaneous realization theorem for the Zariski closure of countably many subsets: 

 \begin{theorem}\label{AdvM}
 {\rm \cite[Theorem 4.1]{Kronecker-Weyl}} Let $G$ be an abelian group with $|G|\leq 2^\mathfrak c$, and let $\mathcal X $ be a countable family of subsets of 
 $G$. Then there exists a (precompact) Hausdorff group topology $\tau$ on $G$ such that the $\tau$-closure of each $X\in \mathcal X $
coincides with its Zariski closure.
 \end{theorem}
 
 In particular, if $\mathcal X $ is a countable family of $\zar$-dense subsets of $G$, then their potential  density is simultaneously witnessed by a single Hausdorff group topology on $G$.
(Every $\zar$-dense subset of an abelian group contains a countable $\zar$-dense subset \cite{Articolone}, so it is worth considering only $\zar$-dense countable subsets.)
In particular, a countable subset $S$ of  an abelian group $G$ with  $|G|\leq 2^\mathfrak c$ is potentially dense if and only if $S$ is $\zar$-dense. In particular, one 
has the following immediate consequence of Theorem \ref{AdvM} providing a complete answer to \cite[Problem 6.5]{TY} (note that \cite[Problem 6.5]{TY} addresses
only almost torsion-free groups, not all groups in $ \M$, as described in Theorem \ref{abelian:Z:cofinite}): 

 \begin{corollary}\label{PvsM}
If $G\in \M$ is abelian, then $G \in \Po$ if and only if $|G|\leq 2^\mathfrak c$.  
 \end{corollary}

Theorem \ref{AdvM} completely resolves Markov's problem on potential density of countable subsets of abelian groups. For the sake of completeness 
we give now a brief account on the remaining part of the problem concerning potential density of uncountable subsets of abelian groups. 
In an abelian group $G$ one has the following immediate necessary condition for potential density of a subset $S \subseteq G$ (see \cite[Corollary 3.2]{DS_HMP}): 
\begin{equation}\label{necessary:condition:equation}
\mbox{$|nG|\le 2^{2^{|nS|}}$  \ for all $n\in\N\setminus\{0\}$.}
\end{equation}

A reinforced version of (\ref{necessary:condition:equation}) turned out to be a sufficient condition for potential density of an uncountable subset $S$: 

\begin{theorem}
\label{sufficient:condition:for:potential:density} {\rm \cite[Theorem 3.4]{DS_HMP}} If $S$ is an uncountable subset of an abelian group $G$ such that $|G|\le 2^{2^{|nS|}}$  for all $n\in\N\setminus\{0\}$, then  there exists a (precompact) Hausdorff group topology $\tau$ on $G$  such that $S$ is dense in $(G,\tau)$. 
\end{theorem}

Since the gap between the above sufficient condition and the necessary condition (\ref{necessary:condition:equation})
completely disappears for groups satisfying $|nG|=|G|$ for every $n\in\N\setminus\{0\}$, for such groups the above theorem becomes a necessary and sufficient condition for potential density of an uncountable subset $S$ (\cite[Corollary 3.6]{DS_HMP}). Let us note that among the groups 
satisfying $|nG|=|G|$ for every integer $n\ge 1$ are all  divisible groups as well as the groups $G$ such that $|t(G)|<|G|$ (in particular, all
almost torsion-free groups). Actually, for groups $G$ with $|t(G)|<|G|$ one has $|nS| = |S|$ for all $n\in\N\setminus\{0\}$, so the necessary and sufficient condition
becomes $|G|\le 2^{2^{|S|}}$, which is obviously necessary, since imposed by the density of $S$ in any Hausdorff topology on $G$.  

\medskip 

This wealth of results in the abelian case leave completely open in the non-abelian context 
the above mentioned problems (potential density, pairs of independent topologies, etc.). For example: 

 \begin{question}\label{Q1}
Does Theorem \ref{abelian:Z:cofinite} remain true in the non-abelian case? 
\end{question}

 \begin{question}\label{Qx}
Do there exist non-abelian groups with a pair of independent topologies? 
\end{question}

 \begin{question}\label{Q2}
Do there exist infinite non-abelian groups in the class $\Po$? In other words, does every infinite non-abelian group of size $\leq 2^\cont$
possess  an infinite non-potentially-dense subset? 
\end{question}

Since $\Po\subseteq \M\subseteq \C$, the answer to Questions \ref{Qx} and \ref{Q2} would be negative if the larger class $\M$ (or even $\C$)
contains no infinite non-abelian groups: 

 \begin{question}\label{Q3}
Do there exist infinite non-abelian groups in the class $\M$? (In other words, does every infinite non-abelian group
possess a proper infinite  unconditionally closed set?)  What about countable groups?  
\end{question}

A negative answer to the first part of this question implies a negative answer to Questions \ref{Qx} and \ref{Q2}. 
Consequenly. 
a group $G$ with a pair of independent topologies is necessarily uncountable and abelian,
hence in Theorem \ref{TY:Ind/Mar-cofin} the assumption of ``abelian" can be omitted (as it will follow from $G \in \M$). 
Similarly, also 
Corollary \ref{PvsM} becomes true without the assumption of ``abelian". We conjecture that 
the answer is negative (at least in the countable case, see Conjecture \ref{MainConj}, the final part of \S \ref{sec:PrincipalResults} and Question \ref{Mcount}).
 
\section{Main results}
  
Inspired by the facts described in \S\S \ref{NoetherianGettho} and \ref{TY-story} and Questions \ref{Q1}--\ref{Q3}, we study in this paper the class $\C$ of groups having cofinite Zariski topology (taking into account that the Zariski topology and the Markov topology coincide for large classes of groups: countable, abelian, free, etc.). We provide an alternative short proof of Theorem \ref{abelian:Z:cofinite} (see Proposition \ref{Exa:Z:cofinite:abel}) and we answer negatively Question \ref{Q1} by showing that Theorem \ref{abelian:Z:cofinite} fails even for nilpotent groups of class 2, see \S \ref{sec:PrincipalResults} and Remark \ref{counterExTY}(c) for more detail. 

In \S \ref{sec:partial zariski} we introduce four partial Zariski topologies by restricting the realm of words appearing in
Definition \ref{Markov's:definition}(a), of those two play a key role 
in a better understanding of the class $\C$. Moreover, 
one of them is used in Theorem \ref{abelian:Z:cofinite} to provide a generalization beyond the abelian context (see Theorem \ref{Super_Main}). In \S \ref{sec:PrincipalResults} the principal results of the paper are collected.  
\subsection{Partial Zariski topologies}\label{sec:partial zariski}

Let $G$ be a group. Taking $x$ as a symbol for a variable, the free product $G[x] = G* \langle x \rangle$ of $G$ and the infinite cyclic group $\langle x \rangle$ 
is \emph{the group of words with coefficients in $G$}. A \emph{word}  $w \in G[x]$  in $G$ has the form
\begin{equation*}              
  w =  g_1 x^{\varepsilon_1} g_2 x^{\varepsilon_2} \cdots g_n x^{\varepsilon_n},
\end{equation*}
for an integer $n \geq 0$, elements $g_1,\ldots, g_n\in G$ and $\varepsilon_1,\ldots,\varepsilon_n \in \{-1,1\}$.

Every word $w \in G[x]$ induces an evaluation function $f_w : G \to G$ defined as follows: for $g \in G$, its image $f_w(g)$ is obtained replacing $x$ in $w$ with $g$. This is why we also write $w(g)$ for $f_w(g)$ sometimes.

 For such $w \in G[x]$, we consider the subset of $G$ defined by
\begin{equation}             \label{first:elem:alg:subs}
E_w^G=f_w ^{-1} (e_G)= \{ g \in G: g_1 g^{\varepsilon_1} g_2 g^{\varepsilon_2} \cdots g_n g^{\varepsilon_n} = e_G\},
\end{equation}
sometimes simply written as $E_w$ in place of $E_w^G$.

If $\mathcal{W} \subseteq G[x]$ is a set of words, we consider the  initial topology $\mathfrak T_\mathcal{W}$ on $G$ of the family $ \{ f_w \mid w \in \mathcal{W} \}$  when considered as functions $G \rightarrow (G, \cof_G)$. A subbase of the closed sets of $\mathfrak T_\mathcal{W}$ is the family of elementary algebraic subsets 
$$\mathcal E( \mathcal{W} ) = \{ f_w ^{-1} (g) \mid w \in \mathcal{W}, g \in G \} = \{ f_{ g^{-1}w } ^{-1} ( e_G ) = E_{ g^{-1}w }\mid w \in \mathcal{W}, g \in G \} \subseteq \mathbb E_G
.$$ 
Such a topology is called a \emph{partial Zariski topology}, as obviously $\mathfrak T_\mathcal{W} \leq \mathfrak Z_G$.

Here we define the four basic examples of partial Zariski topologies used throughout the paper. 

\begin{example}\label{cofin:is:partial:Zariski} The first and simplest example is obtained as follows. 
For every $g \in G$, one has $E_{gx} = \{ g^{-1} \}$. Then the cofinite topology is a partial Zariski topology, as $\cof_G = \mathfrak T_{ \{ gx \mid g \in G \} }$.
\end{example}

\begin{example}\label{mono:topo}[The monomial topology] A word of the form $w = g x^m$, for $g \in G$ and $m\in \Z$, is called a \emph{monomial}. If $\mathcal{M} = \{ g x^n \mid g \in G, \ n \in \N \} \subseteq G[x]$ is the family of the monomials, then we denote by ${\mathfrak Z}_{mon,G}$, or simply by $\pz_{ mon }$, the topology $\mathfrak T_\mathcal{M}$. If $G$ is abelian, then ${\mathfrak Z}_{mon,G} = \zar_G$ (Proposition \ref{Abel:Mon}). 
\end{example}

\begin{example}\label{cen:topo}\label{sec:centralizer:top}[The centralizer topology]
For a non-trivial group $G$, and an element $g \in G$, consider the commutator word $c_g = gxg^{-1}x^{-1} \in G[x]$. The centralizer of $g$ is the subgroup of $G$ defined by $C_G(g) = \{ x \in G \mid xg = gx \} = E_{ c_g }$.

Consider  the family $\mathcal{W}_{cen}= \{x a x^{-1} \mid a \in G\}  = \{ a c_{a^{-1}} \mid a \in G\}  \subseteq G[x]$. For $t\in G$ the set
$f_{x a x^{-1}}  ^{-1}(t)$ is either empty (when $t$ is not a conjugate of $a$) or coincides with $h C_G(a)$ for an appropriate $h\in G$. Since $h C_G(a) = C_G(h a h^{-1}) h$, we have 
\begin{equation}\label{W_cen}
\mathcal E(\mathcal{W}_{cen}) =\{\emptyset\}\cup  \{ h C_G (g) \mid h, g \in G \} = \{\emptyset\}\cup  \{C_G (g')h \mid h, g' \in G \}.
\end{equation}
We call the partial Zariski topology  $\mathfrak T_{ \mathcal{W}_{cen}}$ {\em centralizer topology} on $G$ and we denote it by $\mathfrak C_G$. The pair $(G,\mathfrak C_G)$ is a quasi-topological group, see Proposition \ref{centralizer:top:productive} for further properties of  $(G,\mathfrak C_G)$.
\end{example}

\begin{example}
Since the centralizer topology $\mathfrak C_G$ is not $T_1$ in general (see Lemma \ref{centralizer:top:productive}) we consider also the topology $\mathfrak C_G' = \mathfrak C_G \vee \cof_G$,
where $\vee$ denotes the supremum taken in the lattice $\mathfrak T(G)$ of all topologies on $G$. Note that also $\mathfrak C_G' $ is a partial Zariski topology, as for example 
\[
\mathfrak C_G' = \mathfrak{T}_{ \mathcal{W}_{cen}} \vee \mathfrak T_{ \{ gx \mid g \in G \} } = \mathfrak{T}_{ \mathcal{W}_{cen} \cup \{ gx \mid g \in G \} }.
\]

Clearly, when $\zar_G = \cof_G$ one also has 
\begin{equation}\label{c'=mon=cof}
\mathfrak{C}'_G = \cof_G \text{ and } \mathfrak{Z}_{mon,G} = \cof_G. 
\end{equation}
We are particularly interested in the topologies $\mathfrak{C}'_G$ and $\mathfrak{Z}_{mon,G}$ because all the properties on the structure of $G$ we are able to derive come from the assumption that \eqref{c'=mon=cof} holds. 
\end{example}

The following diagram shows how some of the topologies we study are related, in the lattice of topologies on a group $G$.

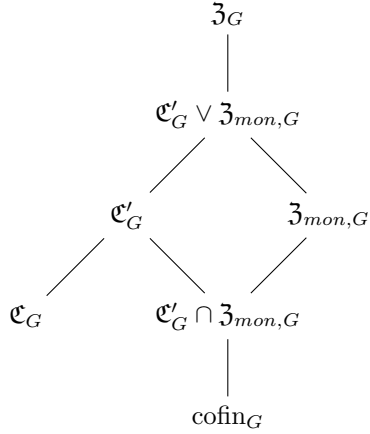
\begin{center}
\begin{tikzpicture}[scale=0.67]
\node(Z) at(0,0)     {$\mathfrak{Z}_G$};
\node(sum)    at (0,-2)     {$\mathfrak{C}_G'\vee \zar_{mon,G}$};
\node(C')  at (-2,-4)     {$\mathfrak{C}'_G$};
\node(C) at (-4,-6)  {$\mathfrak{C}_G$};
\node(monq)    at (2,-4)    {$\mathfrak{Z}_{mon,G}$};
\node(intq) at (0,-6) {$\mathfrak{C}'_G \cap \mathfrak{Z}_{mon,G}$};
\node(cof) at (0,-8) {$\cof_G$};
\draw(Z)  -- (sum) -- (C') -- (intq) -- (monq) -- (sum);
\draw(intq) -- (cof);
\draw(C)   -- (C');
\end{tikzpicture}
\end{center}

We consider groups $G$ with $\mathfrak{Z}_{mon,G} = \cof_G$ in \S\ref{sec:mon-cofinite} (see for example Proposition \ref{mon:cof:atr:or:prime:exp}), and groups $G$ with $\mathfrak{C}'_G = \cof_G$ in \S\ref{sec:structure:C'cofinite} (see for example Theorem \ref{centr:cof:bounded}). We finally consider groups $G$ with $\zar_G = \cof_G$ in \S\ref{zar-cofinite groups} (see also Corollary \ref{Main:prime:exp}). 


\subsection{Principal results}\label{sec:PrincipalResults}

Following Definition \ref{Def:Zar-cofin}, we call \emph{$\mathfrak{C}'$-cofinite} (resp., \emph{${\mathfrak Z}_{mon}$-cofinite}) a group $G$ such that $\mathfrak{C}'_G = \cof_G$ (resp., ${\mathfrak Z}_{mon,G}= \cof_G$).

Since both  $\mathfrak{C}'_G$ and ${\mathfrak Z}_{mon,G}$ are coarser than ${\mathfrak Z}_G$, the class 
$\C_{cen}$ of {$\mathfrak{C}'$-cofinite groups $G$ 
and the class $\C_{mon}$ of ${\mathfrak Z}_{mon}$-cofinite groups both
contain $\C$ and allow us to obtain useful information
about $\C$ (although $\C_{mon} \cap \C_{cen}$ properly contains $\C$, see Example \ref{ex:Tarski:monsters}). 

According to Remark \ref{counterExTY}(c), Theorem \ref{abelian:Z:cofinite} fails in the non-abelian case.
Since ${\mathfrak Z}_{mon} = {\mathfrak Z}_G$ for any abelian group $G$, in the chase of a correct non-abelian version of 
Theorem \ref{abelian:Z:cofinite} we replace, as a first step, $ {\mathfrak Z}_G$ by ${\mathfrak Z}_{mon}$. 
 For the second step we need to recall that a group $G$ is said to \emph{satisfy the cancellation law} if $x^n = y^n$ implies $x = y$, for every $n \in \N_+$ and $x,y \in G$. In particular, such a group is torsion-free. Here we propose a natural weaker property that perfectly fits for our aim of description of the ${\mathfrak Z}_{mon}$-cofinite groups: 

\begin{definition}\label{WCL:def} A group $G$ is said to satisfy the {\em Weak Cancellation Law} (shortly, {\em WCL}) if for every 
$ n > 0$ the map $x \mapsto x^n$ in $G$  is finitely-many-to-one.
\end{definition}

This notion turns out to be quite useful for complete understanding of Theorem \ref{abelian:Z:cofinite} for arbitrary groups, in conjunction with the 
use of ${\mathfrak Z}_{mon}$ in place of  $ {\mathfrak Z}_G$.  
Indeed, in the next theorem we obtain a complete characterization of the class $\C_{mon}$. On the other hand, this is also the desired generalization of Theorem \ref{abelian:Z:cofinite} for arbitrary groups (since when $G$ is abelian, 
${\mathfrak Z}_{mon} = {\mathfrak Z}_G =  {\mathfrak M}_G$ and WLC coincides with ``almost torsion-free").  

\begin{theorem}\label{Super_Main} 
An infinite group $G$ is ${\mathfrak Z}_{mon}$-cofinite if and only if either $G$ has prime exponent, or $G$ is WCL. 
\end{theorem}

In \S \ref{sec:structure:C'cofinite} we focus on infinite non-abelian groups $G\in \C_{cen}$, showing that these groups have no infinite subgroups that are either 
locally solvable or locally finite (Corollary \ref{no:locally:finite}). Moreover, if $G\in \C_{cen}\cap \C_{mon}$, then $G$ has a prime exponent $p \geq 5$ (Theorem \ref{centr:cof:bounded}). 

The productivity of the Zariski topology on groups was studied in \cite{ProductivityZariski}. 
Following that spirit as well as that of Fact \ref{ZNoet ereditaria}, we show in \S \ref{StabProp} that $\C$, $\C_{cen}$, and $\C_{mon}$ are stable under taking subgroups and  taking quotients with respect to finite normal subgroups 
(Corollaries \ref{quotient:zar:cofinite} and \ref{quotient:Mon:cofin} and Lemma \ref{central quotient of C'cof are C'cof}). 
Finite productivity is not available in general, but the situtation can be completely described:  we prove (Theorems \ref{prop :dem} and  \ref{prod:Mon}) that if $G=G_1\times G_2$ is an infinite group, then 
\begin{itemize}
\item[(a)] in case $G$ is non-abelian, then $G\in \C_{cen}$ (resp, $\C$) if and only if exactly one of the groups, say $G_2$, is finite abelian and $G_1\in \C_{cen}$ (resp, $\C$)
 is  infinite non-abelian. 
\item[(b)] $G\in \C_{mon}$ if and only if either both $G_1,G_2$ are WCL, or there exists a prime $p$ such that both $G_1,G_2$ have exponent $p$. 
\end{itemize}

The restriction ``non-abelian" in item (a) is irrelevant, since when $G$ is abelian, then $\zar_G= \zar_{mon,G}$, so 
$G \in \C$ if and only if $G\in \C_{mon}$ (so item (b) applies), while $G\in \C_{cen}$ is vacuously satisfied.  

The following theorem shows us that we can study just countable groups in $\C$, $\C_{cen}$ and $\C_{mon}$.

\begin{theorem}\label{cof:elem:model_intro} An infinite group $G$ is $\mathfrak Z$-cofinite (resp. $\mathfrak{C}'$-cofinite, $\zar_{mon}$-cofinite) if and only if every countable subgroup of $G$ is $\mathfrak Z$-cofinite (resp. $\mathfrak{C}'$-cofinite, $\zar_{mon}$-cofinite).
\end{theorem}
		
The following theorem shows that in studying infinite non-abelian groups in the class $\C$, one can assume several additional properties. 

\begin{theorem}\label{additional:properties} If there exists an infinite non-abelian group $G\in \C$, then there exists an infinite group $K\in \C$ 
such that:
\begin{itemize}
   \item[(a)] $K$ is finitely generated and has prime exponent, 
   \item[(b)] $K$ is perfect, center-free and indecomposable, 
   \item[(c)] $K$ has no proper subgroup of finite index,
   \item[(d)] $K$ has no proper finite normal subgroups.
\end{itemize}
\end{theorem}

Every infinite finitely generated simple group of prime exponent satisfies (b)--(d), in particular so does every Tarski Monster 
(seemingly, none of the Tarski Monsters built so far belongs to $\C$ (\cite{Olshanski2021}, see also Example \ref{ex:Tarski:monsters}).
A careful analysis of the proof of Theorem \ref{additional:properties} (given in \S \ref{zar-cofinite groups}) shows that it works with the 
weaker assumption $G \in \C_{cen} \cap  \C_{mon}$ (giving, of course, as output $K \in \C_{cen} \cap  \C_{mon}$). We expect that the smaller class 
$\C$ contains no infinite non-abelian groups at all (see Conjecture \ref{MainConj}), so the same holds for the smaller class $\M$.

Further properties of the class $\C_{cen}$ are given in \S \ref{Further}. In \S\ref{Engel} we show that infinite non-abelian $\mathfrak{C}'$-cofinite groups
are not Engel groups. In \S \ref{MAX} we collect some nice properties of the maximal finite subgroups of  infinite non-abelian $\mathfrak C'$-cofinite groups.

The final section of the paper offers some open questions. 

\section{Background}

\subsection{Notation and terminology}\label{N&T}

A group $G$ is \emph{center-free} if its center $Z(G)$ is the trivial subgroup. The \emph{$n$-th center} $Z_n(G)$ of $G$ is defined as follows for a positive integer $n$. Let $Z_1(G) = Z(G)$, assume that $n > 1$ and $Z_{n-1}(G)$ is already defined. Consider the canonical projection $\pi \colon G \to G/Z_{n-1}(G)$ and let $Z_n(G) = \pi^{-1} Z (G/Z_{n-1}(G) )$. This produces an ascending chain of characteristic subgroups $Z_n(G)$.
A group is {\em nilpotent} if $Z_n(G) = G$ for some $n$. In this case, its nilpotency class is the minimum such $n$. For example, the groups with nilpotency class 1 are the abelian groups. 

 The commutator of $g,h \in G$ is denoted by $[g,h] = g h g^{-1}h^{-1}$ and the commutator subgroup $G'$ is the subgroup of $G$ generated by all commutators $[g,h]$ of elements $g,h \in G$. Then one can iterate this procedure defining $G^{(2)} = (G')'$ and in general $G^{(n+1)}= (G^{(n)})'$ for $n\in \mathbb{N}$. A group is {\em solvable}  if $G^{(n)} = G$ for some $n$. A group $G$ has nilpotency class 2 if $G$ is non-abelian and $G/Z(G)$ is abelian, i.e. if $\{e\}\ne G' \leq Z(G)$.
 
A group is {\em locally nilpotent (resp. locally solvable, locally finite)} group if every finitely generated subgroup is nilpotent (resp. solvable, finite).

A group $G$ is abelian-by-finite if it has an abelian subgroup of finite index.
Similarly, it is called abelian-by-nilpotent-by-finite group if it has an abelian-by-nilpotent subgroup of finite index, while $G$ is abelian-by-nilpotent if it has a normal abelian subgroup such that the quotient group is nilpotent.

%

\subsection{Some preliminary facts}

We recall a few known properties of nilpotent groups of nilpotency class $\leq 2$, i.e. groups $G$ with $G' \leq Z(G)$.
\begin{fact}		\label{commutator:bilinear:on:nilpo2:groups}
If $G$ has nilpotency class $\leq 2$, then $[ab,g] = [a,g][b,g]$ for every $a, b, g \in G$. In other words, for every $g \in G$ the commutator function $c_g$ mapping $x \mapsto [x,g]$ is a group homomorphism of $G$, with range $G_g = \langle[x,g] \mid x \in G \rangle \leq G' \leq Z(G)$, and with kernel $C_G(g)$.
In particular, $C_G(g)$ contains $G'$, so it is normal in $G$ and $[G:C_G(g)] = |G/C_G(g)] = |G_g|$.
\end{fact}

 An \emph{FC-group} is a group in which every element has Finitely many Conjugates. An immediate consequence of Fact \ref{commutator:bilinear:on:nilpo2:groups} follows.
\begin{corollary}		\label{nilpo2:finite-center}
Let $G$ have nilpotency class $\leq 2$. Then $G$ is an FC-group if and only if $G_g$ is finite for every $g \in G$. 
In particular, if $Z(G)$ is finite then $G$ is an FC-group. 
\end{corollary}
\begin{proof}
To verify that $G$ is an FC-group, let $g \in G$, and we have to check that $[G: C_G(g)]$ is finite. 
For every $g \in G$,  by Fact \ref{commutator:bilinear:on:nilpo2:groups}, it follows that the commutator function $c_g$
is a group homomorphism of $G$, with range contained in $G' \leq Z(G)$, hence finite. As $\ke (c_g) = C_G(g)$, we conclude that $[G: C_G(g)]$ is finite as well.
\end{proof}

The converse implication in the above corollary does not hold in general, i.e. there exist FC-groups of nilpotency class $\leq 2$ that have infinite center. For example, consider the group $G = \Z \times D_8$. Obviously $G$ has nilpotency class $\leq 2$, has infinite center, and the centralizer of a generic element $(n, g)\in G$ is $\Z \times C_{D_8} (g)$, that has finite index in $G$, so that $G$ is FC.


%
%
%


The Zariski topology is clearly a partial Zariski topology. We collect some basic facts about it.

\begin{theorem}{\em \cite[Theorem 5.9]{verbal:functions}}	\label{caratt dei sott norm Z_G chiusi}
Let $N$ be a normal subgroup of a group $G$. Denote by $\overline G = G/N$ the quotient group, and by $\overline {\mathfrak Z}_G$ the quotient topology of $\zar_G$, on $\overline G$.
Then the following conditions are equivalent: 
\begin{itemize}
		\item [(1)] $N$ is $\mathfrak Z_G$-closed;
		\item [(2)] $\overline {\mathfrak Z}_G$ is a $T_1$ topology;
		\item [(3)] $\mathfrak Z_{\overline G} \subseteq \overline {\mathfrak Z}_G$;
		\item [(4)] the canonical map $\pi \colon ( G, \mathfrak Z_G ) \to ( \overline G, \mathfrak Z_{\overline G} )$ is continuous.
	\end{itemize}
\end{theorem}

\begin{proposition}{\em \cite[Corollary 5.10]{verbal:functions}}	
For every group $G$, and every positive integer $n$, the subgroup $Z_n(G)$ is $\mathfrak Z_G$-closed. 
\end{proposition}

\begin{corollary}\label{quotient:zar:cofinite} Let $G$ be a $\mathfrak Z$-cofinite group. Then  every subgroup of $G$ is $\mathfrak Z$-cofinite. 
Moreover, if $N$ is a normal subgroup of a $\mathfrak Z$-cofinite group $G$ and $\overline G = G/N$ then
the following are equivalent: 
\begin{itemize}
    \item [(a)]  $N$ is finite; 
    \item [(b)]  $\mathfrak Z_{\overline G} = \overline {\mathfrak Z}_G = \cof_{\overline G}$, so that also the quotient group $\overline G$ is $\mathfrak Z$-cofinite.
\end{itemize}
\end{corollary}

\begin{proof} Pick a subgroup $H$ of $G$. Then obviously $\mathfrak Z_{H} \leq \mathfrak Z_{G}\restriction_{H}$, so $\mathfrak Z_{H}= \cof_H$, as 
$\mathfrak Z_{G} = \cof_G$, by out hypothesis. The rest easily follows from Theorem \ref{caratt dei sott norm Z_G chiusi}. 
\end{proof}
 
\section{The monomial topology and $\pz_{ mon }$-cofinite groups}\label{sec:mon-cofinite}


\begin{proposition}\label{Exa:Z:cofinite:abel}\label{Abel:Mon}
For an abelian group $G$ one has ${\mathfrak Z}_{mon,G} = \zar_G$ and the following are equivalent:
\begin{itemize}
		\item [(a)] $G$ is almost torsion-free, or $G$ has prime exponent;
		\item [(b)] $G$ is ${\mathfrak Z}_{mon}$-cofinite; 
		\item [(c)] $G$ is ${\mathfrak Z}$-cofinite.
\end{itemize}
\end{proposition}

\begin{proof}
The quality ${\mathfrak Z}_{mon,G} = \zar_G$ is obvious and implies the equivalence between (b) and (c). 
If $G$ is almost torsion-free, or $G$ has prime exponent, then all proper elementary algebraic sets are finite, so $G$ is ${\mathfrak Z}$-cofinite. 
Vice versa, if $G$ is ${\mathfrak Z}$-cofinite, then all proper elementary algebraic sets are finite, so for every positive $n \in \Z$ either $G[n]$ is finite, or $G[n] =G$. If $G[n] =G$ for some $n$, then one can easily see that $G$ has prime exponent. Otherwise, 
$G$ is almost torsion-free. 
\end{proof}

Since the Markov and the Zariski topology coincide on abelian groups, this proposition provides a short immediate proof of Theorem \ref{abelian:Z:cofinite}.

\subsection{Basic properties}

\begin{lemma}\label{zar:mon:hereditary}
If $H$ is a subgroup of $G$, then ${\mathfrak Z}_{mon,H} = {\mathfrak Z}_{mon,G}\restriction_H$.
\end{lemma}
\begin{proof} 
For $g\in G$ and $n \in \Z$ consider $E^G_{gx^n}\cap H=\{x\in H \mid x^n=g^{-1}\}$. If $g\notin H$ then $E_{gx^n}=\emptyset$. Otherwise, when $g\in H$, 
$gx^n\in H[x]$, so $E^H_{gx^n} = E^G_{gx^n}\cap H$, so a subbase of ${\mathfrak Z}_{mon,G}\restriction_H$ is $\{E_{hx^n} \mid  h\in H, n\in \mathbb{Z}\}$.
\end{proof}

Note that $g x \in \mathcal{M}$ for every $g \in G$, so that $\pz_{ mon }$ is a $T_1$ topology, i.e. ${\mathfrak Z}_{mon} \geq \cof_G$, by Example \ref{cofin:is:partial:Zariski}. Call \emph{${\mathfrak Z}_{mon}$-cofinite} a group $G$ satisfying ${\mathfrak Z}_{mon} = \cof_G$. By Lemma \ref{zar:mon:hereditary}, every subgroup of a ${\mathfrak Z}_{mon}$-cofinite group is ${\mathfrak Z}_{mon}$-cofinite.

As $E_{g x^n} = \{ x \in G \mid g x^n = e \} = \{ x \in G \mid x^n = g^{-1} \}$, we describe two cases: if $g = e$, then $E_{g x^n} = E_{x^n} = G[n]$, otherwise $E_{g x^n}$ is the subset of the $n$-th roots of the non-trivial element $g^{-1}$. From this description of the subset $E_{g x^n}$, it follows a characterization of the infinite ${\mathfrak Z}_{mon}$-cofinite groups.
\begin{proposition}\label{mon:cof:characterization}
An infinite group $G$ is ${\mathfrak Z}_{mon}$-cofinite if and only if for every positive integer $n$ the following hold:
\begin{itemize}
\item every non-trivial element has finitely many $n$-th roots,
\item either $G[n]$ is finite, or $G[n] =G$.
\end{itemize}
\end{proposition}
\begin{proof}
The topology ${\mathfrak Z}_{mon}$ is cofinite if and only if for every integer $n > 0$ and for every $g \in G$ the subset $E_{g x^n}$ is either finite, or the whole $G$.
Obviously, $E_{g x^n}$ coincides with $G$ exactly when $g=e$ and $G[n] =G$.
\end{proof}

As an application of Proposition \ref{mon:cof:characterization}, we prove a partial generalization of Proposition \ref{Exa:Z:cofinite:abel}. 

\begin{proposition}\label{mon:cof:atr:or:prime:exp} If a group has prime exponent, then it is ${\mathfrak Z}_{mon}$-cofinite.
An infinite ${\mathfrak Z}_{mon}$-cofinite group either is almost torsion-free, or has prime exponent. 
\end{proposition}

\begin{proof} 
Assume $G$ is a group with prime exponent $p$. Fix a positive integer $n$, for which we show it satisfies the two conditions in Proposition \ref{mon:cof:characterization}. If $p$ divides $n$, then $G[n] = G$, otherwise $G[n] = \{e\}$. Now we check that every non-trivial element $g\in G\setminus \{e\}$ has has at most one $n$-th root in $G$. Indeed, if $p$ divides $n$, then every $x \in G$ satisfies $x^n = e$, so $g \neq e$ has no $n$-th roots in $G$. Otherwise, if $k$ is the multiplicative inverse of $n$ modulo $p$, then $x^n = g$ if and only if $x = g^k$, so $g^k$ is the unique $n$-th root of $g$ in $G$.

Now let $G$ be an infinite ${\mathfrak Z}_{mon}$-cofinite group, and assume it is not almost torsion-free. Then $G[n]$ is infinite for some positive $n$, so $G[n] = G$ by Proposition \ref{mon:cof:characterization}, and $G$ is bounded. We can assume without loss of generality that $G$ has exponent $n$. 
Assume for a contradiction that $n$ is not prime, so there exists a proper factorization  $rs$ of $n$. Then $G[r]\neq G$, so it is finite  by Proposition \ref{mon:cof:characterization}. As $G^s = \{g^s \mid g \in G\} \subseteq G[r]$, there is an $a \in G[r]$ such that $g^s = a$ for infinitely many $g \in G$. In other terms, $E_{ a^{-1} x^s }$ is infinite, hence coincides with $G$. Then $a = e$ and $G=G[s]$, a contradiction.
\end{proof}


\subsection{WCL and almost torsion-free ${\mathfrak Z}_{mon}$-cofinite groups}\label{Sec:WCL}

By Proposition \ref{mon:cof:atr:or:prime:exp}, it remains to understand which almost torsion-free groups are ${\mathfrak Z}_{mon}$-cofinite. The abelian ones are ${\mathfrak Z}_{mon}$-cofinite by Proposition \ref{Exa:Z:cofinite:abel}, and in Proposition \ref{torsion:almosttorsionfree:is:mon-cofinite} we prove that the torsion ones are ${\mathfrak Z}_{mon}$-cofinite.

Obviously, a group satisfying WCL is almost torsion-free, and for abelian groups these two properties are equivalent. 
 Now we show that they are equivalent also for torsion groups.
\begin{proposition}\label{torsion:almosttorsionfree:is:mon-cofinite}
If $G$ is an almost torsion-free group, then the torsion elements of $G$ have finitely many $n$-th roots for every $n > 0$. 
In other words, a torsion group is WCL if and only if it is almost torsion-free.  
\end{proposition}
\begin{proof}
Fix $n>0$. By definition, $G[n]$ is finite and it is the set of $n$-th roots of $e_G$. Let $a \in G$ be an element having finite order $k>0$. Then the set of the $n$-th roots of $a$ is
\[
E = \{ x\in G \mid x^n = a\} \subseteq G[n k]
\]
which is finite. This shows that $G$ is WLC. On the other hand, any group satisfying WCL is almost torsion-free, as mentioned above. 
\end{proof}

The following proposition, connecting WCL with ${\mathfrak Z}_{mon}$-cofiniteness of almost torsion-free groups, 
 was proved in \cite[Corollary 5.21]{verbal:functions} for groups satisfying the cancellation law: 

\begin{proposition}\label{corollary:chernikov}
A group $G$ satisfies the weak cancellation law WCL if and only if $G$ is almost torsion-free and ${\mathfrak Z}_{mon}$-cofinite.
\end{proposition}

\begin{proof} If $G$ is WCL, then $G$ is almost torsion-free, as mentioned above. Hence, it remains to note that for almost torsion-free groups ${\mathfrak Z}_{mon}$-cofiniteness is equivalent to WCL, by Proposition \ref{mon:cof:characterization}. 
%
\end{proof}

Now we can prove Theorem \ref{Super_Main}. Note that it reinforces Proposition \ref{mon:cof:atr:or:prime:exp} by obtaining a complete description of the ${\mathfrak Z}_{mon}$-cofinite groups in analogy with the equivalence of (a) and (b) of Proposition \ref{Abel:Mon} in the abelian case (with almost torsion-free replaced by the (stronger) property WCL). \\

\noindent{\bf Proof of Theorem \ref{Super_Main}.} We have to prove that an infinite group $G$ is ${\mathfrak Z}_{mon}$-cofinite if and only if either $G$ has prime exponent, or $G$ is WCL. This  follows from Propositions \ref{mon:cof:atr:or:prime:exp} and \ref{corollary:chernikov}. \hfill $\Box$



\begin{corollary}
Every nilpotent, torsion-free group is ${\mathfrak Z}_{mon}$-cofinite.
\end{corollary}

\begin{proof} By a classical result due to Chernikov every nilpotent torsion-free group satisfies the cancellation law. 
Since the cancellation law implies WCL, Theorem \ref{Super_Main} applies. 
\end{proof}

We expect that ``torsion-free" in the above corollary  can be relaxed to ``almost torsion-free", in other words {\em nilpotent almost torsion-free groups are ${\mathfrak Z}_{mon}$-cofinite}.  This will immediately follow from Proposition \ref{corollary:chernikov} and the positive answer of Question \ref{Ques:WCL}. 

\begin{corollary}\label{quotient:Mon:cofin} $\C_{mon}$ is stable under taking subgroups and quotients with respect to finite normal subgroups. 
\end{corollary}

\begin{proof} Let $G\in \C_{mon}$. To prove the first assertion pick  a subgroup $H$ of $G$. By Theorem \ref{Super_Main}, either $G$ has prime exponent, or $G$ is WCL. 
Since both properties are inherited by subgroups, $H$ has the same property as $G$, so  $H \in \C_{mon}$, by Theorem \ref{Super_Main}. 

Now let $N$ be a finite normal subgroup of $G$ and $\overline G = G/N$. If $G$ has prime exponent, then $\overline G$ has prime exponent as well, so $G  \in \C_{mon}$, by Theorem \ref{Super_Main}.  Now assume that $G$ is WCL and assume for a contradiction that 
$\overline G$ is not WCL and denote by $q: G\to \overline G $ the quotient map. 
Then for some $b\in \overline G$ and $n>0$ 
 there exist infinitely many solutions $\{y_m\}_{m=0}^\infty$ of $y^n = b$ in $\overline G$. Pick elements $a$ and 
 $\{x_m\}_{m=1}^\infty$ in $G$ such that $q(a) =b$ and $q(x_m) = y_m$ for $m\in \N$. Then $a^{-1}x_m^n \in N$ for all $m\in \N$.
 Since $N$ is finite, there exist $c\in N$ and a subsequence $\{x_{m_k}\}_{k=1}^\infty$ such that $a^{-1}x_{m_k}^n = ac$, so 
 $x_{m_k}^n = a$, a contradiction, since $G$ is WCL. Therefore, $\overline G$ is WCL, so $\overline G \in \C_{mon}$, by Theorem \ref{Super_Main}. 
\end{proof}

\begin{remark}\label{final:cae:mon:cof}
For a group $G$ consider the following conditions: 
\begin{itemize}
\item[(a)] $G$ is ${\mathfrak Z}_{mon}$-cofinite; 
\item[(b)]   ${\mathfrak Z}_{mon,G}$ is a quasi-topological group topology; 
\item[(c)] $G$ is ${\mathfrak Z}_{mon}$-Noetherian. 
\end{itemize}
 Clearly, (a) $\to$ (b) \& (c). On the other hand,  (b) \& (c) does not imply (a) (since every abelian group $G$ is ${\mathfrak Z}_{mon}$-Noetherian, 
 but need not be ${\mathfrak Z}_{mon}$-cofinite; moreover, when $G$ is abelian ${\mathfrak Z}_{mon,G} = {\mathfrak Z}_G$, so it is a quasi-topological group topology). 
\end{remark}

\begin{remark} As corollary of Proposition \ref{mon:cof:characterization} and Proposition \ref{torsion:almosttorsionfree:is:mon-cofinite} 
one can prove that for an almost torsion-free group $G$ the following are equivalent:
\begin{enumerate}
\item $G$ is ${\mathfrak Z}_{mon}$-cofinite;
\item $G$ is WCL, i.e., $\{ x\in G \mid x^n = a\}$ is finite for every $n > 0$ and for every $a \in G$;
\item $\{ x\in G \mid x^p = a\}$ is finite for every prime number $p$ and for every $a \in G \setminus \{e\}$;
\item $\{ x\in G \mid x^p = a\}$ is finite for every prime number $p$ and for every $a \in G$ of infinite order.
\end{enumerate}

Indeed, (1) is equivalent to (2) by Proposition \ref{mon:cof:characterization}; (2) is equivalent to (3) for every group $G$; $(3) \longrightarrow (4)$ is trivial, and $(4) \longrightarrow (3)$ holds by Proposition \ref{torsion:almosttorsionfree:is:mon-cofinite}.
\end{remark}

The next example provides a family of  non-abelian WCL (i.e.,  almost torsion-free groups and $\zar_{mon}$-cofinite) 
 groups that are nilpotent of class 2.

\begin{example}\label{Esempio2} Let $E,F$ and $A$ be abelian almost torsion-free groups and let $\omega:E\times F\longrightarrow A$ be a non-degenerate bilinear form. Let $H=\mathbb{H}(E,F,A)=(E\times F\times A,\cdot)$ be the generalized Heisenberg group defined by
$$(a,b,c)\cdot (x,y,z)=(a+x,b+y,c+z+\omega(a,y))$$
for $a,x\in E$, $b,y\in F$ and $c,z\in A$  \cite{BD,Meg}. Hence the equation
$$(a,b,c)^n=\left( na,nb,nc+\left(\sum_{k=0}^{n}k\right)\omega(a,b)\right)	=(x,y,z)$$
has finitely many solutions for every $n\in \mathbb{Z}$, $a,x\in E$, $b,y\in F$ and $c,z\in A$. By Proposition \ref{mon:cof:characterization}, $H$ is $\zar_{mon}$-cofinite group of nilpotency class $2$. If any of the groups $E,F$ or $A$ is not torsion then $H$ is not torsion.

A concrete choice for $E,F$ and $A$ can be:

(a) $E=F=A=D$, where $D$ is any domain with $char D = 0$, and $\omega(x,y)=xy$ the multiplication in $D$ (in this case we briefly write $\mathbb{H}_D$,
in place of $\mathbb{H}(D,D,D)$);

(b) $E$ is a discrete abelian group, $F = \widehat E$ its Pontryagin dual, $A=\T$ (the circle group) and $\omega(x,\chi) = \chi(x)$, for 
$(x,\chi) \in E\times F$. To have 
$H=\mathbb{H}(E,F,A)$ almost torsion-free, it is necessary that $E$ and $F= \widehat E$ are 
almost torsion-free. This occurs precisely when $r_p(E)<\infty$ and  $|E/pE|<\infty$ for every prime $p$. 
For example one can take an almost torsion-free divisible $E$, say $E= \bigoplus_{j=1}^m\Z(p_j^\infty)$, where 
$p_j$ are not necessarily distinct  primes. Since $\T$ is almost torsion-free. this necessary condition is also sufficient (see above).   
\end{example}

To provide a generalization of  Example \ref{Esempio2} we need the following lemma which let us understand why the above example works:    

\begin{lemma}\label{WCLxcentral:extensions}
Let $G$ be a group such that both $Z(G)$ and $G/Z(G)$ satisfy WCL. Then $G$ satisfies WCL.   
\end{lemma}

\begin{proof}
For $a\in G$ and $n\in \N_+$ consider the equation $x^n = a$. Let $q: G \to G/Z(G)$ be the canonical homomorphism. 
 Then the equation $y^n = q(a)$ has finitely many solutions $y_1, \ldots, y_m$. Pick and fix 
 $$
 x_1, \ldots, x_m\in G\ \mbox{ with }\ q(x_i) = y_i  \mbox{ for } i= 1, \ldots, m.
 $$ 
 Since $q(x_i^n)=q(x_i)^n = y_i^n = q(a)$, we deduce that $x_i^n = u_ia$ for some $u_i \in Z(G)$. 
 Consider now a solution $x$ of $x^n = a$. Since $y=q(x)$ is a solution of $y^n = q(a)$,   there exists $i= 1, \ldots, m$
 such that $q(x) = y_i$. Hence, for some $z\in Z(G)$ one has $x=zx_i$. This $z$ satisfies $a= x^n =z^nx_i^n = z^nu_ia$, so  $z^n =u_i^{-1}$. Since $z, u_i \in Z(G)$ and $Z(G)$ satisfies WCL, one has finitely many elements $\xi_{i,j}$, $j= 1,\ldots, k_i$, of $Z(G)$ with $\xi_{i,j}^n= u_i^{-1}$. 
The  finite set $F= \{\xi_{i,j}: i= 1, \ldots, m; j = 1,\ldots, k_i\} \subseteq Z(G)$ has the property that for every  solution $x$ of $x^n = a$ 
belongs to the finite set $\bigcup_{i=1}^m Fx_i$.
%
\end{proof}

Now we can generalize Example \ref{Esempio2}.

\begin{corollary} Let $G$ be a nilpotent group of class 2, such that both $Z(G)$ and $G/Z(G)$ are almost torsion-free. Then $G$ satisfies WCL, hence $G \in \C_{mon}$.
\end{corollary}

\begin{proof}
Since almost torsion-free abelian groups satisfy WCL, both $Z(G)$ and $G/Z(G)$ satisfy WCL. Now Lemma \ref{WCLxcentral:extensions} applies. 
\end{proof}

One can extend this corollary to an arbitrary nilpotent group $G$, by imposing all quotients $Z_i(G)/Z_{i-1}(G)$ to be almost torsion-free. 
%
%
%

\section{The centralizer topology} 

Here come some properties of  $(G,\mathfrak C_G)$.

\begin{lemma}
\label{centralizer:top:productive}
\begin{enumerate}
\item \cite[Lemma 4.5]{secondo:survey} Let $\{G_i \mid i \in I\}$ be a family of groups, and $G = \prod_{i \in I} G_i$. Then $\mathfrak C_G = \prod_{i\in I} \mathfrak C_{G_i}$;
\item \cite[Proposition 5.30]{verbal:functions} If $H \leq G$, then $\mathfrak C_{H} \leq \mathfrak C_{G}\restriction_{H}$;
\item $Z(G)$ is the $\mathfrak{C}_G$-closure of $\{e\}$;
\item $(G,\mathfrak C_G)$ is a quasi-topological group, and it is $T_1$ if and only if $G$ is center-free.
\end{enumerate}
\end{lemma}

\begin{proof}
3. Since $a C_G (g) \ni e$ precisely when $a C_G (g) = C_G (g)$, the $\mathfrak{C}_G$-closure of $\{e\}$ coincides with $Z(G) = \bigcap_{g\in G}C_G(g)$. 

4. Obviously translations and inversion are continuous due to the properties of the family (\ref{W_cen}) of Example \ref{cen:topo}. This proves that $(G,\mathfrak C_G)$ is a quasi-topological group.  

According to 3., if $(G,\mathfrak C_G)$ is $T_1$, then $G$ is center-free. 
On the other hand, if $Z(G)$ is trivial, then $\{e\} = Z(G) = \bigcap_{g\in G}C_G(g)$ is $\mathfrak C_G$-closed. 
Since $(G,\mathfrak C_G)$ is a quasi-topological group, all singletons are $\mathfrak C_G$-closed.
\end{proof}

\begin{example}\label{zariski topology on a free group}
(a) If $F$ is a free non-abelian group, then $\mathfrak C_F = \mathfrak C_F' = \mathfrak Z_F$ by \cite[Theorem 5.35]{verbal:functions}.  
As $F$ has infinite centralizers, this topology is not cofinite, so $F$ is $\zar$-Noetherian and not $\zar$-cofinite. On the other hand, 
free groups  satisfy the cancellation law, so WCL as well. Thus  $\mathfrak Z_{mon,F} = \cof_F$. 

(b) If $G=\mathbb H_K$, for a field $K$ of characteristic 0, as in Example \ref{Esempio2}(a), then $\mathfrak C_G' \ne \mathfrak C_G' = \mathfrak Z_G$
(\cite[\S 5.4.2]{verbal:functions}), while $\mathfrak Z_{mon,G} = \cof_G$, by Theorem \ref{Super_Main} and Example \ref{Esempio2}. 

\end{example}

\subsection{$\mathfrak C'$-cofinite groups}\label{sec:structure:C'cofinite}

Every abelian group $G$ satisfies $\mathfrak C_G' = \cof_G$, as well as every finite group. Here comes a family of examples of $\mathfrak C'$-cofinite groups. 

\begin{example}\label{ex:Tarski:monsters}
Let $p$ be a fixed prime number. A \emph{Tarski Monster group for $p$} is an infinite group $G$ all of whose nontrivial subgroups have $p$ elements, so in particular $G$ has exponent $p$. In $1979$, Ol$'$shanskij showed that such groups exist for every prime $p$ that is big enough (see his papers \cite{Ol'shanski79, Ol'shanski81}).
By the first part of Proposition \ref{mon:cof:atr:or:prime:exp}, Tarski monster groups are ${\mathfrak Z}_{mon}$-cofinite. 

 All Tarski monster groups are obviously $\mathfrak C'$-cofinite, 
 while those constructed so far are non-topologizable  (\cite{Olshanski2021}), so ${\mathfrak M}$-discrete. But since every Tarski monster $T$ is countable, one 
 has $\zar_T=\mar_T$, so these Tarski monsters $T$ are also ${\mathfrak Z}$-discrete, hence $T \not \in \C$. Therefore, they
witness the fact that $\mathfrak C'_T = {\mathfrak Z}_{mon,T} = \cof_T$ does not imply $\zar_T= \cof_T$, i.e., the class $\C_{cen}\cap \C_{mon}$ 
properly contains $\C$. 
\end{example}

 We collect some immediate consequences of Lemma \ref{centralizer:top:productive}. In the sequel, for $g \in G$, we briefly denote by $w_g$ the word $g^x = x^{-1} g x \in G[x]$.
    
\begin{proposition}\label{fact:CG:cofinite}
Let $G$ be an infinite non-abelian $\mathfrak C'$-cofinite group. Then the following hold:
\begin{itemize}
  \item[(1)] $Z(G)$ is finite;
  \item[(2)] $G$ has no infinite abelian subgroups; 
  \item[(3)] if $g \in G \setminus Z(G)$, then $C_G(g)$ is finite, so $[G:C_G(g)]$ is infinite and $g$ has infinitely many conjugates;
  \item[(4)] every subgroup $H$ of $G$ satisfies $\mathfrak C_H \leq \mathfrak C_G \restriction _H \leq \cof_G \restriction _H = \cof _H$ by Lemma  \ref{centralizer:top:productive}(2), so $H$ is $\mathfrak C'$-cofinite;
  \item[(5)] $G$ is torsion;
  \item[(6)] the commutator function $f_{ c_g } : (G, \mathfrak C_G) \to (G, \mathfrak C_G)$ and the function $f_{w_g} \colon (G, \mathfrak C_G) \to (G, \mathfrak C_G)$ are continuous for every $g \in G$. 
\end{itemize}
\end{proposition}
\begin{proof}  (1).  Follows  immediately from the definition. 

(2). Assume that $H$ is an infinite abelian subgroup of $G$. By (1),  there exists $h \in H \setminus Z(G)$.
Then $C_G(h)\ne G$ must be finite, as $G$ is $\mathfrak C'$-cofinite. Since obviously $H \leq C_G(h)$, this contradicts 
our assumption that $H$ is infinite. 

(3) is immediate and (4) follows from (3). 

(5). The group $G$ cannot contain infinite cyclic subgroups $H$. Indeed, such an $H$ is $\mathfrak C'$-cofinite, by (3). Hence, $Z(H)$ must be finite, by (1).
This contradicts the fact that $H$ is abelian. The assertion follows also  from the fact that every element $g \in G$ either lies in $Z(G)$, which is finite, or $C_G(g) \neq G$, hence $C_G(g)$ is finite. 

(6). It is enough to prove only the second assertions, since $f_{ c_g }$ is a composition of $f_{w_g }$ and a left translation, by $g^{-1}$. 
To this aim it is sufficient to prove that $f_{w_g }^{-1}(h)$ is $\mathfrak C_G$-closed for every $h \in G$. 
Fix an element $h \in G$ and note that $w = {x g x^{-1}}\in \mathcal{W}_{cen}$ (the family (\ref{W_cen}) introduced in Example \ref{sec:centralizer:top}). Hence, $f_{w}^{-1}(h)$ is $\mathfrak C_G$-closed. By Lemma \ref{centralizer:top:productive}(4), the pair $(G, \mathfrak C_G) $ is a quasi-topological group, so the inversion $\iota : (G, \mathfrak C_G) \to (G,\mathfrak C_G)$ is homeomorphism. Therefore $f_{w_g }^{-1}(h) = \iota(f_{w}^{-1}(h))$ is $\mathfrak C_G$-closed as well. 
\end{proof}

\begin{remark}\label{counterExTY}
(a) For an infinite non-abelian $\mathfrak C'$-cofinite group and $g\in G \setminus Z(G)$, the centralizer $C_G(g)$ is finite and contains $Z(G)$, so $C_G(g)$ is a finite union of cosets of $Z(G)$. As a consequence, the family of $\mathfrak C_G$-closed sets consists of $G$ and finite unions of cosets of $Z(G)$, i.e. $\mathfrak C_G$ is the initial topology of the canonical map $G \to (G/Z(G), \cof_{ G/Z(G) })$.

(b) By Proposition \ref{fact:CG:cofinite}(3), an infinite non-abelian $\mathfrak C'$-cofinite group strongly fails to be an FC-group.

(c) The almost torsion-free (actually, WCL) group $H$ in Example \ref{Esempio2} has infinite abelian subgroups, so $H \not \in \C_{cen}$ by Proposition \ref{fact:CG:cofinite}(2), and consequently $H \not \in \C$. This shows that Theorem \ref{abelian:Z:cofinite} fails in the non-abelian case (even for nilpotent groups of class 2). 
\end{remark}


Obviously, every central subgroup is normal. For finite subgroups of an infinite $\mathfrak C'$-cofinite group  we prove now a stronger version of the converse if this property. 
 

\begin{proposition}\label{fin:normalizer}
Let be a $H$ be a finite subgroup of an infinite $\mathfrak C'$-cofinite group $G$. Then either $N_G(H)$ is finite, or $H\leq Z(G)$. In particular, $H$ is normal if and only if it is central.
\end{proposition}
\begin{proof} 
As $G$ is $\mathfrak C'$-cofinite, $f_g \colon (G, \mathfrak C_G) \to (G, \mathfrak C_G)$ is continuous for every $g \in G$, by  Proposition \ref{fact:CG:cofinite}(5). 

The normalizer $N_G(H)$ of a subgroup $H$ of $G$ is the maximum among the subgroups of $G$ in which $H$ is normal, so
\begin{equation}\label{eq:normalizer}
N_G(H) = \{x \in G \mid \forall g \in H \quad g^x \in H \} = \bigcap_{g \in H} \{x \in G \mid f_g(x) \in H \}  = \bigcap_{g \in H} f_g^{-1} (H).
\end{equation}

As $H$ is obviously closed in $\mathfrak C_G$, equation (\ref{eq:normalizer}) shows that $N_G(H)$ is closed in $\mathfrak C_G$ too, so either $N_G(H)$ is finite, or $N_G(H) = G$. In the second case, $H$ is normal, so each of its elements $h\in H$ has only finitely many conjugates in $G$, as they all lie in the finite group $H$. Then $C_G(h)$ has finite index in $G$, so that $C_G(h) =G$ and $h \in Z(G)$.
\end{proof}

For the sake of convenience we give separately a particular case of Proposition \ref{fin:normalizer} for $\zar$-cofinite groups. 

\begin{corollary}\label{finite:centralizer}
A finite subgroup $H$ of an infinite $\zar$-cofinite group $G$ 
is either normal, or $N_G(H)$ is finite. Hence, $H$ is normal if and only if it is central.
\end{corollary}

Now we prove that the central ascending series $Z(G) \leq Z_2(G) \leq Z_3(G) \leq \ldots$ of a $\mathfrak C'$-cofinite non-abelian group $G$ stabilizes after at most two steps, i.e. the second center $Z_2(G)$ coincides with the first center $Z(G)$.

\begin{theorem}\label{z2=z1}\label{no:solvable}
If $G$ is an infinite non-abelian $\mathfrak C'$-cofinite group. Then the following hold true.

(a) $Z_2(G) = Z(G)$. In particular, $G$ has no infinite  nilpotent subgroups. 

(b) $G^{(n)}$ is infinite for every $n$. In particular, $G$ has no infinite solvable subgroups.
\end{theorem}
\begin{proof} (a) Let $g \in G$. The commutator function $f_{ c_g } : (G, \mathfrak C_G) \to (G, \mathfrak C_G)$ is continuous, by Proposition \ref{fact:CG:cofinite}(5).  

By definition, $Z_2(G)$ is the preimage under the canonical projection $G \to G/Z(G)$ of the center $Z\bigl(  G/ Z(G)  \bigr)$. Equivalently, 
\begin{equation}\label{def:Z2}
Z_{ 2 } (G) = \bigcap_{g \in G} \{ x \in G \mid [g,x] \in Z(G) \} = 
									\bigcap_{g \in G} f_{ c_g } ^{-1} ( Z(G) ).
\end{equation}
As $Z(G)$ is $\mathfrak C_G$-closed by Lemma \ref{centralizer:top:productive}(3), equation (\ref{def:Z2}) yields that $Z_2(G)$ is $\mathfrak C_G$-closed as well.

Assume that $Z_2(G)$ is infinite, so $Z_2(G) = G$ and $G$ has nilpotency class $2$. Then
Corollary \ref{nilpo2:finite-center} applies, and $G$ is an FC-group -- this contradicts Remark \ref{counterExTY}(c). Then $Z_2(G)$ is finite. 
Proposition applies \ref{fin:normalizer} to $Z_2(G)$, 
to conclude that $Z_2(G) \leq Z(G)$, hence $Z_2(G) = Z(G)$.

Now, if $H$ is an infinite subgroup of $G$, then $H$ is $\mathfrak C'$-cofinite (by Proposition \ref{fact:CG:cofinite}(4)) and  
non-abelian, by Proposition \ref{fact:CG:cofinite}(2). Then the above argument gives $Z_2(H) = Z_1(H) \neq H$, so $H$ is not nilpotent. 

(b) To prove the first assertion we argue by induction. 
For $n=1$ the subgroup $G'$ is not central, since $G$ cannot have nilpotency class $2$, by item (a). As $G'$ is normal, 
we conclude that $G'$ is infinite, by Proposition \ref{fin:normalizer}. Moreover, $G'$ is infinite $\mathfrak C'$-cofinite (by  Proposition \ref{fact:CG:cofinite}(3))
and non-abelian, by Proposition \ref{fact:CG:cofinite}(2). This concludes the case $n=1$. The same argument allows us to conclude that $G^{(n+1)}$ is infinite, non-abelian
and $\mathfrak C'$-cofinite, whenever $G^{(n)}$ is. This proves the first assertion. In particular, $G$ is not  solvable. 

Since every infinite subgroup $H$ of $G$ is $\mathfrak C'$-cofinite and is not abelian by Proposition \ref{fact:CG:cofinite}(2), we deduce by the above 
argument that $H$ is not solvable. 
\end{proof}

By  Proposition \ref{fact:CG:cofinite}(5), an infinite non-abelian $\mathfrak C'$-cofinite group is torsion. Now we prove that $G$ is bounded, of prime exponent bigger than $3$.

\begin{theorem}\label{centr:cof:bounded}\label{zar-cofinite:prime:exp}
If $G$ is an infinite non-abelian $\mathfrak C'$-cofinite group, then $G$ is not almost torsion-free. Moreover, if $G$ is also ${\mathfrak Z}_{mon}$-cofinite, then it has a prime exponent $p \geq 5$.
\end{theorem}
\begin{proof}
Since $G$ is non-abelian, there exists $g \notin Z(G)$. By Proposition \ref{fact:CG:cofinite}(5), $g$ has finite order, say $n$, and then all the infinitely many conjugates of $g$ have also order $n$, so that $G[n]$ is infinite. In particular, $G$ is not almost torsion-free. 
%

Now assume that $G$ is also ${\mathfrak Z}_{mon}$-cofinite. Then we can apply Proposition \ref{mon:cof:atr:or:prime:exp} to deduce that the exponent of $G$ is a prime number $p$, since 
$G$ is not almost torsion-free. Obviously $p \neq 2$ as $G$ is not abelian. 

It is an old observation of Levi and van der Waerden that in a group of exponent $3$ every element commutes with all its conjugates. Indeed, if $g^3=e_G$ for every $g \in G$, and $x,y \in G$, then
\begin{multline*}
xyxy^{-1} = (xy )^{-2} xy^{-1} = ( y^{-1}x^{-1} y^{-1}x^{-1} ) xy^{-1} = y^{-1}x^{-1}y^{-2} = y^{-1} x^{-1} y = \\y^{-1} (x^{-1} y )^{-2} 
= y^{-1} ( y^{-1} x y^{-1} x ) =y^{-2} xy^{-1} x = yxy^{-1} x.
\end{multline*}
Taking a non-central element $g$ of $G$, we see that it has infinitely many conjugates, while $C_G(g)$ is finite, so that $p \neq 3$.
\end{proof}

As a consequence of Proposition \ref{mon:cof:atr:or:prime:exp} and Theorem \ref{centr:cof:bounded}, for an infinite non-abelian $\mathfrak C'$-cofinite group $G$ the two following properties are equivalent:
\begin{itemize}
\item $G$ is ${\mathfrak Z}_{mon}$-cofinite;
\item $G$ has a prime exponent $p \geq 5$.
\end{itemize}

\begin{corollary}\label{Main:prime:exp} 
If $G$ is an infinite non-abelian $\zar$-cofinite group, then $G$ has a prime exponent $p \geq 5$.
\end{corollary}

Here follows a particular case of Corollary \ref{finite:centralizer}, applied to the finite non-central cyclic subgroup $H = \langle g \rangle$ (see below)
with a reinforced conclusion:

\begin{corollary}
Let $G$ be an infinite non-abelian $\zar$-cofinite group, and $g\in G\setminus Z(G)$. Then $N_G( \langle g \rangle ) = C_G(g)$. 
\end{corollary}
\begin{proof}  Let $N = N_G(\langle g \rangle)$. By Corollary \ref{Main:prime:exp}, $G$ has a prime exponent $p$, while Corollary \ref{finite:centralizer} implies that $N$ is finite, so $N$ is a finite $p$-group.  The action of $N$ on $\langle g \rangle$ by conjugation gives an homomorphism $\phi \colon N \to Aut (\langle g \rangle)$, with kernel $ker \ \phi = C_G(g) = C$. Since $N$ is a finite $p$-group, the same holds true for $N/C$ as well, so $|N/C| $ is a power of $p$. On the other hand, $Aut (\langle g \rangle)$ is isomorphic to the cyclic group of order $p-1$, so $|N/C| \mid p-1$, hence $|N/C| = 1$. Consequently, $N_G( \langle g \rangle ) = C_G(g)$.
\end{proof}

\begin{corollary}\label{cor:example:not:zar:cofinite}
The infinite non-abelian groups in Example \ref{abeliani e lineari sono ZNoet} are not $\mathfrak C'$-cofinite.
\end{corollary}

\begin{proof} By Theorem \ref{no:solvable}(b), every solvable-by-solvable-by-finite $\mathfrak C'$-cofinite group is either finite or abelian, so in particular this proves our statement for the groups in Example \ref{abeliani e lineari sono ZNoet}(b)-(c).

Let $G$ be an infinite non-abelian linear group, and assume by contradiction that $G$ is $\mathfrak C'$-cofinite. By the Tits alternative \cite{Tits}, $G$ has a subgroup $H$ that is either:
\begin{itemize}
\item solvable, normal, and of finite index in $G$,
\item or isomorphic to the free group on two generators.
\end{itemize} 
In particular, $H$ is infinite. The first case contradicts Theorem \ref{no:solvable}(b), while the second case contradicts Theorem \ref{centr:cof:bounded}.
\end{proof}

\begin{corollary}\label{no:locally:finite}
Let $G$ be a non-abelian $\mathfrak{C'}$-cofinite group. The infinite subgroups of $G$ are neither locally finite, nor locally solvable.
\end{corollary}

\begin{proof} Infinite locally finite groups have an infinite abelian subgroup by the Hall-Kulatilaka-Kargapolov Theorem (for a proof see \cite[14.3.7]{Robinson}). By 
Proposition \ref{fact:CG:cofinite}(2), $G$ has no such subgroups. 

As far as the second assertion is concerned, torsion locally solvable groups are locally finite. By 
 Proposition \ref{fact:CG:cofinite}(5) and the first part of the proof, $G$ has no infinite locally solvable subgroups.
\end{proof}

By a celebrated result of Zelmanov \cite[Theorem 2]{Z}, every compact torsion group is locally finite. By Theorem \ref{centr:cof:bounded} and Corollary \ref{no:locally:finite}, a non-abelian $\mathfrak{C'}$-cofinite group does not admit any compact Hausdorff group topology (this should be compared to Question \ref{question:hausdorff:top}).

%

\subsection{
 Stability properties of $\C$, $\C_{cen}$, and $\C_{mon}$}\label{StabProp}

 We already saw that $\C$, $\C_{cen}$, and $\C_{mon}$ are stable with respect to taking subgroups and we saw (Corollaries \ref{quotient:zar:cofinite} and \ref{quotient:Mon:cofin}) that  $\C$ and $\C_{mon}$ are stable also under taking quotients with respect to finite normal subgroups. Now we check that also  $\C_{cen}$ has the latter property, recalling that  finite normal subgroups in this class are central, by Proposition \ref{fin:normalizer}. 
 
\begin{lemma}\label{central quotient of C'cof are C'cof} Let $G\in \C_{cen}$ and $H\leq Z(G)$. Then $G/H\in \C_{cen}$. \end{lemma}

\begin{proof} Let $gH$ be a non-central element of $G/H$, so that in particular $g \notin Z(G)$. Let $h\in H$. Since $G$ is $\mathfrak{C}'$ cofinite then $[x,g]=h$ has finitely many solutions. The subgroup $H$ is finite, thus 
\[
\{x\in G \mid[x,g]\in H\}=\bigcup_{h\in H}\{x\in G \mid [x,g]=h\}
\]
is finite, as the map $c_g$ is finitely many-to-one. As $C_{G/H}(gH)=\{xH\in G/H \, |\, [x,g]\in H\}$, we conclude that also $C_{G/H}(gH)$ is finite.
\end{proof}

 The following is the counterpart of Theorem \ref{Product:Z:Noeth:iff:}(b) for the classes $\C$ and $ \C_{cen}$.  

\begin{theorem}\label{prop :dem} Let $G=G_1\times G_2$ be an infinite non-abelian group. Then $G\in \C_{cen}$ (resp, $\C$) if and only if exactly one of the groups, say $G_2$ is finite abelian and $G_1\in \C_{cen}$ (resp, $\C$) is infinite non-abelian. 
\end{theorem}
\begin{proof}

$(\Rightarrow)$ Assume that $G = G_1 \times G_2\in \C_{cen}$. Then $Z(G) = Z(G_1) \times Z(G_2)$ is finite, and let $\bar G_i = G_i / Z(G_i)$ for $i=1,2$.
By Proposition \ref{z2=z1}, the group $\bar G = \bar G_1 \times \bar G_2 \cong G/Z(G)$ is 
center-free.
Moreover, $\bar G\in \C_{cen}$ by Lemma \ref{central quotient of C'cof are C'cof}. Then Lemma \ref{centralizer:top:productive}(1) gives
\begin{equation} \label{eq:di:m}
\mathfrak C_{\bar G_1} \times \mathfrak C_{\bar G_2} = \mathfrak C_{\bar G} = \mathfrak C_{\bar G}' = \cof_{\bar G}.
\end{equation}

As $\bar G$ is infinite, at last one of the groups $\bar G_i$ to be infinite, assume that $\bar G_1$ is infinite. As $\bar G_1$ is center-free, $G_1$ is infinite non-abelian. Then 
\[
C_{\bar G}(\{ e_{ \bar G_1 } \} \times \bar G_2 )= \bar G_1 \times Z( \bar G_2 )
\]
is infinite, hence it coincides with $\bar G$. Then $\bar G_2$ is abelian, so $\mathfrak C_{\bar G_2}$ is indiscrete. Then (\ref{eq:di:m}) gives that $\bar G_2$ is trivial, so $G_2$ is abelian, hence finite by Theorem \ref{no:solvable}. 

 Now assume that $G = G_1 \times G_2\in \C$. Since $\C \subseteq \C_{cen}$, the above argument implies that one of the groups, say $G_2$, is finite abelian
and $G_1\in \C_{cen}$. Actually, $G_1\cong G/G_2\in \C$, by Corollary \ref{quotient:zar:cofinite}.  

$(\Leftarrow)$  Assume that $G_1\in \C_{cen}$ is infinite and non-abelian, and $G_2$ is a finite abelian group.
If $g = (g_1,g_2) \in G$, then $C_G (g)= C_{G_1} (g_1) \times C_{G_2}(g_2) = C_{G_1} (g_1) \times G_2$. Then either $C_{G_1} (g_1) = G_1$, and so $C_G (g)=G$, or $C_{G_1} (g_1)$ is finite, and so $C_G (g)$ is finite. Hence, $G\in \C_{cen}$.

Now assume that $G_1\in \C$ is infinite non-abelian, and $G_2$ is a finite abelian group.
By \cite[Theorem 3.4]{ProductivityZariski}, $\Zar_G \leq \Zar_{G_1}\times \Zar_{G_2}$. Since, our hypothesis $G_1\in \C$ entails $\Zar_{G_1}= \cof_{G_1}$, we have 
$$\cof_G \leq \Zar_G \leq \Zar_{G_1}\times \Zar_{G_2} = \cof_{G_1} \times  \cof_{G_2} = \cof_{G}, 
$$
since $G_2$ is finite. Therefore, $G \in \C$. 
\end{proof}

Next we describe finite productivity in $\C_{mon}$. If $G_1\times G_2\in \C_{mon}$, then both $G_1,G_2\in \C_{mon}$, by Corollary \ref{quotient:Mon:cofin}. 
Nevertheless, this implication cannot be inverted in general as we see below. 

\begin{theorem}\label{prod:Mon} 
Let $G=G_1\times G_2$ be an infinite group. Then $G\in \C_{mon}$ if and only if either both groups $G_1,G_2$ are WCL, or there exists a prime $p$ such that both groups $G_1,G_2$ have exponent $p$. 
\end{theorem}
\begin{proof} 
If $G\in \C_{mon}$, then either $G$ has prime exponent, or $G$ is WCL, by Theorem \ref{Super_Main}. Since both properties are inherited by subgroups, 
we conclude that either both groups have exponent $p$ or  both groups are WCL. 
The sufficiency of this condition is obvious, since both properties (WCL and having exponent $p$) are finitely productive.  
\end{proof}

 The following example shows that infinite productivity strongly fails for all three classes $\C_{mon}, \C_{cen}, \C$. 

\begin{example} We show below that $\C_{mon},  \C_{cen}$ and $\C$ are not stable under countably infinite direct products.  

(a) The group $G = \Q/\Z$ is almost torsion-free, so $G \in \C_{mon}$, yet $G^\N$ is not almost torsion-free, so 
$G^\N \not \in \C_{mon}$ since it does not have prime exponent. 

(b) Let $G_0$ be a finite non-abelian group and $\{e\} \ne G_n\in \C_{cen}$ for every $n\geq 1$. Then $G= \prod_{n=0}^\infty \not \in \C_{cen}$,
since the centralizer of every $g\in G_0 \setminus Z(G_0)$ is a proper infinite subgroup of $G$. Taking every $G_n$ to be a non-trivial finite group shows that neither $\C_{cen}$ nor $\C$ are stable under countably infinite direct products.  
\end{example}

A group is called \emph{indecomposable} if it cannot be expressed as the direct product of two proper normal subgroups. Now we show that in studying the infinite non-abelian $\mathfrak C'$-cofinite groups, it is sufficient to consider the indecomposable ones.

\begin{proposition}\label{new:center-free:impl:indec}
 If $G\in \C_{cen}$ (resp., $G\in \C$) is infinite and non-abelian, then $G \cong G_1 \times G_2$, where $G_1\in \C_{cen}$ (resp., $G_1\in \C$) is infinite, non-abelian and indecomposable group, while  $G_2$ is a finite (possibly trivial) abelian group.
\end{proposition}

 \begin{proof} If $G$ is indecomposable, there is nothing to prove. Otherwise, let $G \cong I_1 \times A_1$, and by Theorem \ref{prop :dem} we can assume $I_1$ to be infinite non-abelian, and $A_1$ to be non-trivial finite abelian. 

If $I_1$ is indecomposable, we are done. Otherwise, as $I_1$ is $\mathfrak C'$-cofinite, we can apply Theorem \ref{prop :dem} to $I_1$ to obtain that $G \cong I_2 \times A_2 \times A_1$, with $I_2$ infinite non-abelian, and $A_2$ finite non-trivial abelian. As $G$ does not contain infinite abelian subgroups by Proposition \ref{fact:CG:cofinite}(2), it cannot contain infinite strictly increasing chains of abelian subgroups. Therefore, this process stops after finitely many steps.
 \end{proof}
 
 
 In an infinite group $G\in \C_{cen}$, a finite subgroup is normal if and only if it is central, by Proposition \ref{fin:normalizer}. Now we obtain a property of the infinite normal subgroups of $G$ as another application of Theorem \ref{prop :dem}. 

\begin{proposition}\label{infinite:intersection}
Let $G\in \C_{cen}$ be infinite and non-abelian. If $H_1$, $H_2$ are infinite normal subgroups of $G$, then $H_1 \cap H_2$ is infinite.
\end{proposition}
\begin{proof}
Assume by contradiction that $H_1 \cap H_2$ is finite, so that $H_1 \cap H_2 \leq Z(G)$ by Proposition \ref{fin:normalizer}. 
%
%
Consider the quotient group $\bar G = G/Z(G)$, and the canonical map $\pi :G \to \bar G$. For $i=1,2$, if $K_i = H_i Z(G)$, then $\bar K_i = \pi(K_i) = K_i/Z(G)$ is an infinite normal subgroup of $\bar G$. As $\bar G$ is $\mathfrak C'$-cofinite by Lemma \ref{central quotient of C'cof are C'cof}, Proposition \ref{fin:normalizer} implies that $\bar{K_1}\cap \bar{K_2} \leq Z( \bar G)$. As $\bar G$ is center-free by Proposition \ref{z2=z1}, we deduce that $\bar{K_1}\cap \bar{K_2}$ is trivial. Then the infinite subgroup $\langle\bar{K_1}, \bar{K_2}\rangle$ of $\bar G$ is isomorphic to the direct product $\bar{K_1} \times \bar{K_2}$, contradicting Theorem \ref{prop :dem}, as both $\bar{K_1}$ and $\bar{K_2}$ are non-abelian by Theorem \ref{no:solvable}. 
\end{proof}

\begin{remark} For an infinite and non-abelian $G\in \C_{cen}$ and an infinite normal subgroup $N$ of $G$ one has $Z(N)=N\cap Z(G)$.
Indeed, $N\cap Z(G)\leq Z(N)$. By Proposition \ref{fact:CG:cofinite}, $N$ is $\zar$-cofinite, hence $Z(N)$ is finite. Then $Z(N)$ is a finite normal subgroup of $G$ since it is a characteristic subgroup of a normal subgroup of $G$. Then $Z(N)\leq Z(G)$ by Proposition \ref{fin:normalizer}.
\end{remark}

\section{Proof of Theorems \ref{cof:elem:model_intro} and  \ref{additional:properties}}\label{countable:zar-cofinite}\label{zar-cofinite groups}


Clearly, $G$ is $\zar$-cofinite if and only if every proper elementary algebraic subset $E_w$ is finite.

\begin{lemma}\label{Lem:countableSub} 
Let $G$ be an infinite group, $w \in G[x]$, and assume that $E_w^G\subsetneq G$ is infinite. Then there exists a countable subgroup $H$ such that $w \in H[x]$ and $E_w^H\subsetneq H$ is infinite.
\end{lemma}

\begin{proof} Let $X$ be a countable subset of $E_w^G$, and $g \in G \setminus E_w^G$, and consider the subgroup $H$ of $G$ generated by $X$, the coefficients of $w$, and $g$. Then $H$ is countable, $w \in H[x]$, and $E_w^H = E_w^G \cap H$ is a proper (witnessed by $g \in H \setminus E_w^H$) and infinite (witnessed by $X \subseteq E_w^H$) subset of $H$. 
\end{proof}

In \cite[Theorem 5.4]{secondo:survey} (see also Theorem \ref{Product:Z:Noeth:iff:}(a)), we have proved that a group is $\mathfrak Z$-Noetherian if and only if all its  countable subgroups are $\mathfrak Z$-Noetherian. Now we prove an analogue for $\zar$-cofinite groups, for $\mathfrak{C}'$-cofinite groups and for $\zar_{mon}$-cofinite groups. 

\bigskip

\noindent {\bf Proof of Theorem \ref{cof:elem:model_intro}.} $(\Rightarrow)$ Every infinite subgroup of a $\mathfrak{Z}$-cofinite (resp. $\mathfrak{C}'$-cofinite, $\zar_{mon}$-cofinite) 
group is $\mathfrak{Z}$-cofinite (resp. $\mathfrak{C}'$-cofinite, $\zar_{mon}$-cofinite).

$(\Leftarrow)$ Assume that $G$ is not $\mathfrak Z$-cofinite. Then $E_w^G \subsetneq G$ is infinite for some $w \in G[x]$. By Lemma \ref{Lem:countableSub}, $G$ has a countable subgroup $H$ that is not $\mathfrak{Z}$-cofinite. 

If $G$ is not $\mathfrak{C}'$-cofinite, there is $g\in G$ such that the subgroup $C_G(g)$ is proper and infinite. Let $w=[g,x] \in G[x]$ and apply Lemma \ref{Lem:countableSub} to $E_w^G = C_G(g)$ to get a countable subgroup $H$ such that $g \in H$ and $C_H(g)\neq H$ is infinite. Hence, $H$ is not 
$\mathfrak{C}'$-cofinite. 

Similarly, if $G$ is not $\zar_{mon}$-cofinite, then $E_{gx^n}^G\ne G$ is infinite for some $g\in G$ and $n\in \mathbb{N}$. We can apply again Lemma \ref{Lem:countableSub} to $w=g x^n \in G[x]$ to get a countable subgroup $H$ such that $g \in H$ and $E_w^H\ne H$ is infinite. Hence, $H$ is not $\zar_{mon}$-cofinite. 

A second argument in last case can be given, based on Theorem \ref{Super_Main}. Indeed, this theorem implies that if $G$ is not $\zar_{mon}$-cofinite, then $G$ is not WCL and $G$ does not have prime exponent. Then there exist a countable subgroup $H_1$ of $G$ that is not WLC and a  countable subgroup $H_2$ of $G$ that is not of prime exponent. Then the countable subgroup $H$ generated by $H_1$ and by $H_2$ is  neither WCL nor of prime exponent. Therefore, $H$ is not $\zar_{mon}$-cofinite. \hfill $\Box$

\medskip

 With a similar proof one can obtain more. If $\mathcal{W} \subseteq G[x]$ is a countable  family of words, the above theorem easily generalizes as follows: 
if $\mathfrak T_\mathcal{W} \ne \cof_G$, then there is a countable $H\leq G$ with $\mathcal{W} \subseteq H[x]$ and $\mathfrak T_\mathcal{W} \ne \cof_H$. 

%
%

As a consequence of Theorem \ref{cof:elem:model_intro}, we can limit the study of the classes $\C, \C_{cen}, \C_{mon}$  to the case of countable groups.
As anticipated in Theorem \ref{additional:properties}  we can restrict our attention to finitely generated $\mathfrak Z$-cofinite groups. 

\bigskip

Now we prepare the proof of Theorem \ref{additional:properties}. To this end we recall that for every prime $p$ the finite $p$-groups are nilpotent, while infinite $p$-groups need not be even solvable. 

In the next lemma we give some additional conditions under which the derived series $G' \geq G'' \geq \ldots$ stabilizes, after finitely many steps, to a non-trivial subgroup.

For a group $G$ we denote by $r_{fin}(G)$ the intersection 
\[
r_{fin}(G) = \bigcap_{N\triangleleft G, \, [G:N]< \infty} N = \bigcap_{H\leq G, \, [G:H]< \infty} H.
\] 
Clearly, $G/r_{fin}(G)$ is residually finite; in particular, $G$ is residually finite if and only if $r_{fin}(G)$ is trivial.

\begin{lemma}\cite{Kos}\label{fin:gen:prime:exp}
 Let $G$ be an infinite finitely generated non-abelian group of prime exponent. Then $G^{(m)} =G^{(m+1)}$ for some positive integer $m\in \N$, and this subgroup 
 of $G$ is finitely generated and has finite index in $G$. In particular, $G^{(m)}=r_{fin}(G)$. 
\end{lemma}
\begin{proof}
Let $p$ be the prime exponent of $G$. The abelianization $G/G'$ is a finitely generated abelian group of exponent $p$, hence it is finite (and it is an elementary abelian $p$-group indeed). Since finite index subgroups of finitely generated groups are themselves finitely generated (\cite[Proposition 2.1]{Mann}), we deduce that $G'$ is finitely generated.  Proceeding by induction we obtain that the index $[G^{(n)}:G^{(n+1)}]$ is finite for every $n\in \N$. 

Now we prove that the subgroup $H = r_{fin}(G)$ coincides with $G^{(m)}$ for some $m\in \mathbb{N}$.
We use the fact that $G/H$ is finitely generated, residually finite and it has exponent $p$. Let us see that this implies that 
$G/H$ is finite. To this end we use the following well-known claim.  
Nevertheless, we provide a complete proof for the benefit of the reader. 

\begin{claim} An infinite finitely generated group $G$ of positive exponent is not residually finite. 
\end{claim}

\noindent {\em Proof of the Claim.}
According to the Restricted Burnside problem (see \cite{Kos} and \cite{Zelmanov, Zelmanov2} for the solution of this celebrated problem.) the finite quotients of a
 finitely generated group $G$ of finite exponent are, up to isomorphism, finitely many, say $B_1, …, B_s$. Put $B = \prod_{l=1}^s B_i$. 
Then, if such a group $G$ were residually finite, then it would be isomorphic to a subgroup of a power $P =B^I$, for some  infinite index set $I$.
Let us see that this assumption leads to a contradiction. This follows from more general results in the realm of varieties of algebras (see \cite[Theorem 3.49]{Bergman}), 
but as mentioned above we prefer to give a self-contained argument. 

Let $r = |B|$ and for $J\subseteq I$ denote by $\Delta_J$ the diagonal subgroup of $B^J$. Clearly, $\Delta_J \cong B$ is finite. 
Assume that $ G = \langle g_1, …, g_k\rangle$ and consider each generator $g_i$ as a function $g_i: I \to B$. This gives rise to a finite partition $\mathcal P_i:= \{I_{i,j}\}_{j=1,…,r}$ of $I$ (with possibly $I_{i,j} = \emptyset$), such that the function $g_i: I \to B$ is constant on $I_{i,j}$.

Taking all non-empty intersections $\bigcap_{i=1}^sI_ {i,j_i}$, where $(j_1,\ldots,j_s)\in \{1,\ldots,r\}^k$ we obtain a new finite partition $\{J_l\}_{l=1}^L$  of $I$, that is finer than all these partitions $\mathcal P_{i}$. For $l = 1,\dots , L$ all generators $g_i$ are simultaneously constant functions on $J_l$,
so  the corresponding projection $p_l: B^I\to B^{J_l}$ sends all $g_i$ in the diagonal subgroup $\Delta_{{J_l}}$ of $B^{J_l}$. 
 Therefore, $p_l(G)$ is a finite subgroup of $B^{J_l}$ for all $l = 1,\dots , L$. Since $G$ is isomorphic to a subgroup of $B^I = \prod_{l=1}^L B^{J_l}$, such that all its projections $p_l(G)$ are finite, we deduce that $G$ is finite.    \hfill $\Box$\\

Now we can finish the proof of Lemma \ref{fin:gen:prime:exp}, concluding with the above claim that  $G/H$, being finitely generated, residually finite and of exponent $p$, must be finite, hence nilpotent.   If $m\in \mathbb{N}$ is the nilpotency class of $G/H$, then $G^{(m)}\leq H$. On the other hand, $G^{(m)}$ has finite index in $G$, so $H\leq  G^{(m)}$.
\end{proof}

A group $G$ is called {\em perfect} if $G'=G$. Clearly, a non-trivial finite $p$-group is never perfect.  

\begin{lemma}\label{caseG=G'}
 A perfect $p$-group has no proper subgroups of finite index.
\end{lemma}
\begin{proof}
If $G$ has a proper subgroup of finite index, then $G$ has a proper normal subgroup of finite index, say $H$. Then $G/H$ is a finite perfect $p$-group, that is a contradiction.
\end{proof}

\noindent {\bf Proof of Theorem \ref{additional:properties}.} We have to prove that if there exists an infinite non-abelian group $G \in \C$, then there exists also 
an infinite group $K \in \C$ such that:
\begin{itemize}
    \item[(a)] $K$ is finitely generated and has prime exponent, 
    \item[(b)] $K$ is perfect, center-free and indecomposable, 
    \item[(c)] $K$ has no proper subgroups of finite index,
    \item[(d)] $K$ has no proper finite normal subgroups.
\end{itemize}

Assume that $G$ is an infinite non-abelian $\zar$-cofinite group. Then $G$ has prime exponent $p$ by Theorem \ref{zar-cofinite:prime:exp}. By Corollary \ref{no:locally:finite}, $G$ is not locally finite, so $G$ has an infinite finitely generated subgroup $S$, which is $\zar$-cofinite itself. Then $S$ is non-abelian by Proposition \ref{fact:CG:cofinite}(2). As $S$ has exponent $p$, we can apply Lemma \ref{fin:gen:prime:exp} to find an integer $m$ such that the subgroup $H=S^{(m)}$ of $S$ is perfect, and has finite index in $S$. In particular, $H$ is infinite, 
finitely generated, and has exponent $p$.


(a)-(b)-(c). By Theorem \ref{z2=z1}(a), $Z_2(H) =Z(H)$ is finite. So $K = H/Z(H)$ is infinite and center-free. 
Being a quotient of $H$, also $K$ is finitely generated, perfect, and has exponent $p$. Then $K$ has no proper finite-index subgroup by Lemma \ref{caseG=G'}. On the other hand, $K$ is indecomposable by Proposition \ref{new:center-free:impl:indec}. 

(d). Apply Corollary \ref{quotient:zar:cofinite} to the $\zar$-cofinite group $H$ and its finite subgroup $Z(H)$, so that $\mathfrak Z_{K} = \overline {\mathfrak Z}_H = \cof_K$, we conclude that $K$ is also $\zar$-cofinite. In particular, $K$ has no proper finite normal subgroups by Proposition \ref{fin:normalizer}.	
\hfill $\Box$

The above argument works with the weaker hypothesis $G \in \C_{cen}\cap  \C_{mon}$, and under this weaker assumption the conclusion will give a group $K\in \C_{cen}\cap  \C_{mon}$ with the properties (a)--(d).

\section{Further properties of the class $\C_{cen}$}\label{Further}

\subsection{Engel elements of $\mathfrak{C}'_G$-cofinite groups}\label{Engel}

Let $G$ be a group. For elements $x,y\in G$, let us define by induction
\[
[x,_0 y]=x,\quad [x,_{n+1} y]=[[x,_n y],y].
\]
We can define the subset of {\em Engel elements} of $G$ as:
\begin{align*}
L(G)&=\{x\in G\, |\, \text{ for every } g\in G \text{ there exists }n\in\N \text{ such that }[x,_n g]=e\}.
\end{align*}
In particular, $Z(G)\subseteq L(G)$.

A group $G$ is called an \emph{Engel group} if $G=L(G)$. If there exists $n\in \mathbb{N}$ such that $[x,_n y]=e$ for every $x,y\in G$ then $G$ is said to be an $n$-Engel group. Clearly if $G$ is nilpotent of class $n$ then it is $n$-Engel. Conversely, finite Engel groups are nilpotent. 

The {\it Fitting subgroup} $F(G)$ of a group $G$ is the subgroup generated by all nilpotent normal subgroups of $G$. The Fitting subgroup and the set $L(G)$ are related by the following theorem.
\begin{theorem}\label{group_gen_by_n_engel}\label{Peng}\cite{Peng}
If $G$ is a group satisfying the maximal condition on abelian subgroups, then $L(G)=F(G)$.
\end{theorem}

\begin{theorem}
Let $G$ be an infinite non-abelian $\mathfrak{C}'$-cofinite group. Then $Z(G)=L(G)=F(G)$. In particular, $G$ is not an Engel group.
\end{theorem}
\begin{proof}
According to Theorem \ref{no:solvable}(a), the nilpotent subgroups of $G$ are finite. Thus, the finite normal nilpotent subgroups of $G$ are central according to Proposition \ref{fin:normalizer}. Therefore $F(G)\leq Z(G)$. The group $G$ satisfies the hypothesis of Theorem \ref{group_gen_by_n_engel}, indeed $G$ has no infinite abelian subgroups. Then we have
\[
Z(G)\subseteq L(G)=F(G)\leq Z(G).		\qedhere
\]
\end{proof}

\subsection{Maximal finite subgroups of $\mathfrak C'$-cofinite groups}\label{MAX}

 By definition, in a Tarski Monster group the finite cyclic subgroups are maximal subgroups. In this section we show that infinite non-abelian $\mathfrak C'$-cofinite groups have maximal finite subgroups, and we prove some additional properties that infinite non-abelian $\mathfrak C'$-cofinite groups share with Tarski monster groups. 

We start proving that maximal finite subgroups exist in an infinite non-abelian $\mathfrak C'$-cofinite groups.

\begin{proposition}\label{maximal finite sub}
Let $G$ be an infinite non-abelian $\mathfrak C'$-cofinite group. Then every finite subgroup of $G$ is contained in a 
maximal finite subgroup $M$ of $G$, and such $M$ satisfies
\[
Z(G) \leq Z(M) = C_G(M) \leq N_G(M)=M.
\]
\end{proposition}
\begin{proof} Let $H$ be a finite subgroup of $G$. The union of an ascending chain of finite subgroups of $G$ is locally finite, thus finite by Corollary \ref{no:locally:finite}.
 Hence, every ascending chain of finite subgroups of $G$  containing $H$ stabilizes. Therefore, maximal finite subgroups containing $H$ exist.

Let $M$ be a maximal finite subgroup of $G$. As $Z(G)$ is finite, the subgroup $\langle M, Z(G)\rangle$ is finite and contains $M$, so $Z(G)\leq M$.

If the normalizer $N_G(M)$ of $M$ is finite, then $N_G(M)=M$ by maximality of $M$. 
Otherwise, Proposition \ref{fin:normalizer} yields that $M$ is central, so $M = Z(G)$. 
Take an element $g \in G \setminus Z(G)$ and put $H = \langle g, Z(G)\rangle$. Then $Z(G) \lneq H$. By  Proposition \ref{fact:CG:cofinite}(5), $G$ is torsion, so  $ \langle g \rangle$ is finite. As $H$ is isomorphic to a quotient of the finite group $\langle g \rangle \times Z(G)$, we deduce that $H$ is finite.
This contradicts the maximality of $M$, since the finite subgroup $H$ properly contains  $M = Z(G)$.

Obviously $C_G(M)\leq N_G(M)=M$ and so $Z(M) = C_G(M)\cap M = C_G(M)$.
\end{proof}

\begin{proposition}\label{intersection of finite is central2}
Let $G$ be an infinite non-abelian $\mathfrak C'$-cofinite $p$-group, where $p$ is a prime number. Then $M_1\cap M_2 = Z(G)$ for every pair of maximal finite subgroups $M_1\neq M_2$ of $G$. 
\end{proposition}

\begin{proof}
By Proposition \ref{maximal finite sub}, we have $M_1\cap M_2 \geq Z(G)$. To prove that $M_1\cap M_2 \leq Z(G)$, without loss of generality we can assume that $M_2$ has also the additional property that $I=M_1\cap M_2$ is maximal among the subgroups of $G$ of the form $M_1 \cap M$, as $M\neq M_1$ ranges among the maximal finite subgroups of $G$.

By Proposition \ref{fin:normalizer}, either $I\leq Z(G)$, or $N_G(I)$ is finite, and we shall exclude the latter case. So assume by contradiction $N_G(I)$ to be finite, and let $M$ be a maximal finite subgroup containing $N_G(I)$. Then obviously $M$ contains both $N_{M_1}(H)$ and $N_{M_2}(H)$ and we have the following diagram in the lattice of subgroups of $G$: 

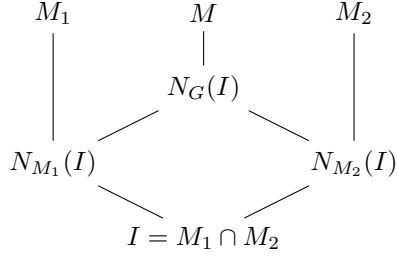
\bigskip 
\begin{center}
\begin{tikzpicture}
\node(M_1) at(-2,1)     {$M_1$};
\node(M) at(0,1)     {$M$};
\node(M_2) at(2,1)     {$M_2$};
\node(N_1) at(-2,-1)     {$N_{M_1}(I)$};
\node(N) at(0,0)     {$N_{G}(I)$};
\node(N_2) at(2,-1)     {$N_{M_2}(I)$};
\node(I) at(0,-2)     {$I=M_1\cap M_2$};
\draw(M_1)  -- (N_1);
\draw(M)  -- (N);
\draw(M_2)  -- (N_2);
\draw(N)  -- (N_1);
\draw(N)  -- (N_2);
\draw(N_1)  -- (I);
\draw(N_2)  -- (I);
\end{tikzpicture}
\end{center}

\bigskip 

As $M_1$ is a finite $p$-group and $I \neq M_1$, one has $N_{M_1}(I) \gvertneqq I$, and so
\begin{equation}\label{Max}
M_1\cap M\geq M_1\cap N_{G}(I) \geq M_1\cap N_{M_1}(I) = N_{M_1}(I) \gvertneqq I=M_1\cap M_2.
\end{equation}
Now $M_1\cap M  \gvertneqq M_1\cap M_2$ and the special choice of $M_2$ ensuring the maximality of $M_1\cap M_2$ implies that $M=M_2$ . Therefore (\ref{Max}) gives now  $M_1\cap M_2\gvertneqq M_1\cap M_2$, a contradiction.
\end{proof}

In Proposition \ref{intersection of finite is central2}, the assumption on $G$ to be a $p$-group can be relaxed to the following one: every finite subgroup of $G$ is nilpotent (which is equivalent to: every finite subgroup of $G$ has the normalizer condition).

\begin{corollary}\label{cor:maximal}
Let $G$ be an infinite non-abelian $\mathfrak C'$-cofinite $p$-group, where $p$ is a prime number. Then every element $g \in G \setminus Z(G)$ in contained in a unique maximal finite subgroup $M_g$ of $G$, and $N_G( \langle g \rangle ) \leq M_g$.
\end{corollary}
\begin{proof}
By Proposition \ref{fin:normalizer}, $N_G( \langle g \rangle )$ is finite, and let $M$ be a maximal finite subgroup of $G$ containing $N_G( \langle g \rangle )$. 

If $g \in \widetilde{M}$ for a maximal finite subgroup $\widetilde{M}$ of $G$, then $\langle g \rangle \leq M \cap \widetilde{M}$, so that $M = \widetilde{M}$ by Proposition \ref{intersection of finite is central2}.
\end{proof}

\begin{corollary}
 Let $p$ be a prime and $G$ be an infinite non-abelian $\mathfrak C'$-cofinite $p$-group. 
 If $x,y \in G$ do not belong to the same maximal finite subgroup of $G$, then $H=\langle x,y\rangle$ is infinite and $Z(H) = Z(G) \cap H$. In particular,
   $|Z(H)|=1$, provided $|Z(G)|=1$. 
\end{corollary}

\begin{proof} If $H$ were finite, we could find a maximal finite subgroup $M$ containing $H$, by virtue of Proposition \ref{maximal finite sub}.
Since $M$ contains both $x$ and $y$, this will contradict our hypothesis that $x$ and $y$ do not belong to the same maximal finite subgroup of $G$.
Moreover,  
\[
Z(H) = Z(\langle x,y\rangle) = C_G(x)\cap C_G(y)\cap H \leq N_G(x)\cap N_G(y)\cap H.
\]
Applying respectively Corollary \ref{cor:maximal} and Proposition \ref{intersection of finite is central2} we obtain
$Z(H)\leq M_x\cap M_y \cap H = Z(G) \cap H$. Since obviouosly $Z(H) \geq Z(G) \cap H$, we obtain $Z(H) = Z(G) \cap H$.
\end{proof}

\section{Final comments and open question}


 \begin{problem}
 Describe the potentially dense subsets of a free non-abelian group $F$. Is it true that 
 every $\Zar$-dense subset of $F$ is potentially dense? 
\end{problem}

Recall that $\Mar_F= \Zar_F=\mathfrak C_F$, by the recent result of  Shakhmatov--Yañez \cite{SY} and  Example \ref{zariski topology on a free group}, so the $\Zar$-dense (i.e., $\Mar$-dense)  subsets of $F$ are the subsets that are not contained in a finite union of cosets of cyclic subgroups of $F$.

 \begin{problem}
 Describe the potentially dense subsets of the Heisenberg group $\mathbb H_\Q$. Is it true that 
 every $\Zar$-dense subset of $\mathbb H_\Q$ is potentially dense? 
\end{problem}

Recall, that $\Zar_{\mathbb H_\Q}=\mathfrak C_{\mathbb H_\Q}$, by Example \ref{zariski topology on a free group}.

\begin{question}\label{Mcount}
 Does the counterpart of Theorem \ref{cof:elem:model_intro} for the class $\M$ remain true? 
\end{question}

A positive answer to this question, along with Theorem \ref{cof:elem:model_intro}, would imply that $\M = \C$, due to Markov's theorem on coincidence of
Markov and Zariski topologies for countable groups. On the other hand, a counter-example $G$ will be a group with $\Zar_G \ne \Mar_G$ (so providing 
a negative answer to Markov's Problem \ref{Markov's:problem*}). Indeed, $\Mar_G \ne \cof_G$, while $\Mar_H = \cof_H$ for every countable subgroup $H$
of $G$. This implies $\Zar_H = \cof_H$ for every countable subgroup $H$ of $G$, so $G \in \C$, by Theorem \ref{cof:elem:model_intro}. 
Now $\Zar_G = \cof_G \ne \Mar_G$ entails the desired inequality $\Zar_G \ne \Mar_G$. 


\begin{question}\label{Ques:WCL:bis} {\rm \cite[Problem 6.4]{TY}, \cite[Problem 14]{BT}} Does every uncountable $\Zar$-cofinite Abelian group admit two independent Hausdorff group topologies?
\end{question}

By Theorem \ref{TY:Ind/Mar-cofin}, Question \ref{Ques:WCL:bis} has a positive answer for Abelian groups of size $\cont$, under the assumption of MA. The rigid restriction on the  size is imposed by the approach used in \cite{TY}, based on the following stronger result of independent interest: every second countable group topology on an Abelian  ${\mathfrak Z}$-cofinite group of size $\cont$ 
admits an independent group topology \cite[Theorem 5.2]{TY}. Moreover, the topologies independent with a second countable one are necessarily 
countably compact (\cite[Proposition 2.4]{TY}), this imposes $|G|\geq \cont$ 
(since an infinite countably compact groups has size $\geq \cont$). On the other hand, 
every second countable Hausdorff group has size $\leq \cont$, so $|G| = \cont$. 

\medskip

It is easy to check that $( E_{g x^n} )^{-1} = E_{g^{-1} x^n}$, yet it is not clear whether the traslate of a set $E_{g x^n}$ is ${\mathfrak Z}_{mon}$-closed, and this leaves open the following question: 

\begin{question}\label{Mon_vs_QTOP} Does there exist a group $G$ such that $(G,{\mathfrak Z}_{mon,G})$ is not a quasi-topological group?
\end{question}

For a group $G$ the poset $\mathfrak Q(G)$ of all quasitopological group topologies on $G$, as a subposet of the complete lattice  $\mathfrak T(G)$ of all topologies on the set $G$, is a complete lattice and its complete lattice operations coincide with those of $\mathfrak T(G)$ (although this fails to be true for the smaller poset $\mathfrak L(G)$ of all group topologies on $G$, see \cite{ADB}).  In particular, for every $\tau\in \mathfrak T(G)$ there is a coarsest $\tau^q\in \mathfrak Q(G)$ finer than $\tau$, namely $\tau^q = \inf \{\sigma \in \mathfrak Q(G): \sigma \geq \tau\}$.  
So, $\mathfrak{Z}_{mon,G} \leq \mathfrak{Z}_{mon,G}^q \leq \zar_G$. 
Obviously, the ${\mathfrak Z}_{mon}$-cofinite and the ${\mathfrak Z}_{mon}^q$-cofinite groups coincide, this is why we preferred to focus on ${\mathfrak Z}_{mon}$-cofiniteness, since ${\mathfrak Z}_{mon}$ is easier to deal with.

\begin{question}\label{Ques:WCL}
Are nilpotent almost torsion-free groups always WCL? 
\end{question}

\medskip

Recall that the groups from $\C_{cen}\cap \C_{mon}$ have prime exponent, by Proposition \ref{mon:cof:atr:or:prime:exp}.
\begin{question} 
Let $G\in \C_{cen}\cap \C_{mon}$ be  infinite and non-abelian. 
\begin{itemize}
  \item Is $G$ a Tarski monster group?
  \item If not: are maximal finite subgroups of $G$ cyclic? 
  \item If not: are maximal finite subgroups abelian? are the cardinalities of maximal finite subgroups bounded? 
  \item Are all maximal finite subgroups conjugated?
\end{itemize}
What if in addition $G\in \C$?
\end{question}

\begin{question}\label{Ques:x}
How many conjugacy classes of elements are there in  an infinite non-abelian $\mathfrak C'$-cofinite group? 
What about $\zar$-cofinite ones?
\end{question}

Proposition \ref{infinite:intersection} shows that  the family $\mathcal B$ of all infinite normal subgroups 
of an infinite non-abelian $\mathfrak C'$-cofinite group $G$ is a filter base, so gives rise to a group topology 
$\mathcal T_G$ on $G$ with $\mathcal B$ as a local base at $e$.  

\begin{question}\label{question:hausdorff:top}
Is $\mathcal T_G$ is a Hausdorff topology?
\end{question}

A group $G$ is called an \emph{FNI-group} (for Finite Normalizer Index) if for every subgroup $H$ of $G$ 
\begin{equation}\label{eq:FNI}
\text{either } [N_G(H):H]< \infty, \text{ or $H$ is normal in $G$.}
\end{equation} 
By Proposition \ref{fin:normalizer}, finite subgroups of a $\mathfrak C'$-cofinite group satisfy \eqref{eq:FNI}. We do not know if remains true for all subgroups of $G$:

\begin{question}
Are the groups in $\C_{cen}$ necessarily FNI? 
\end{question}

We conclude with our main conjecture: 

\begin{conjecture}\label{MainConj}
Every infinite group in $\C$ is necessarily abelian, so the same holds for the smaller class $\M$.
\end{conjecture}

A positive answer to this conjecture would imply a positive answer to Question \ref{Mcount} and the 
second part of Question \ref{Ques:x}, as well as a negative answer to Questions \ref{Q2} and \ref{Q3}.


\end{document}